\documentclass[11pt]{article}
\usepackage{amsmath,amssymb,amsthm}
\usepackage[all]{xy}

\newcommand{\pb}{{\boldsymbol{p}}}

\newcommand{\bone}{{1,\ldots,1}}
\newcommand{\cI}{{\cal I}}

\newcommand{\C}{{\mathbb C}}
\newcommand{\D}{{\mathbb D}}
\newcommand{\N}{{\mathbb N}}

\newcommand{\R}{{\mathbb R}}
\newcommand{\T}{{\mathbb T}}
\newcommand{\Z}{{\mathbb Z}}

\newcommand{\calS}{\mathcal{S}}
\newcommand{\sos}{\mathrm{sos}}

\theoremstyle{plain}
\newtheorem{Theorem}{Theorem}[section]

\newtheorem{Corollary}[Theorem]{Corollary}

\newtheorem{Lemma}[Theorem]{Lemma}
\newtheorem{Proposition}[Theorem]{Proposition}

\theoremstyle{definition}

\newtheorem{Remark}[Theorem]{Remark}
\newtheorem{Example}[Theorem]{Example}

\newcommand{\ol}{\overline}

\renewcommand{\subset}{\subseteq}

\newcommand{\footnoteremember}[2]{
  \footnote{#2}
  \newcounter{#1}
  \setcounter{#1}{\value{footnote}}
} \newcommand{\footnoterecall}[1]{
  \footnotemark[\value{#1}]
}

\begin{document}

\title{An Algebraic Perspective on Multivariate Tight Wavelet Frames. II}

\author{
Maria Charina\footnoteremember{myfootnote}{Fakult\"at f\"ur
Mathematik, TU Dortmund, D--44221 Dortmund, Germany}, \ Mihai
Putinar \footnote{Mathematics Department, University of California at Santa Barbara,
Santa Barbara CA 93106, $mputinar@math.ucsb.edu$, and
DMS-SPMS, Nanyang Technological University, 21
Nanyang Link, Singapore 637371, $mputinar@ntu.edu.sg$},\\
Claus Scheiderer \footnote{Fachbereich Mathematik und Statistik,
Universit\"at Konstanz, D--78457 Konstanz, Germany} \ and \
Joachim St\"ockler \footnoterecall{myfootnote} }

\maketitle

\begin{abstract}Continuing our recent work in \cite{CPSS2012}
we study polynomial masks of multivariate tight wavelet frames from two additional and complementary points of view: convexity and system theory. We consider such polynomial masks that
are derived by means of the unitary
extension principle from a single polynomial.
We show that the set of such polynomials is convex and reveal its
extremal points as polynomials that satisfy the quadrature mirror filter condition.
Multiplicative structure of such polynomial sets allows us to improve the known upper bounds on the number of
frame generators derived from box splines.
In the univariate and bivariate settings, the polynomial masks of a tight wavelet frame
can be interpreted as
the transfer function of a
conservative multivariate linear system. Recent advances in system theory
enable us to develop a more effective method for tight frame constructions.
Employing an example by S. W. Drury, we show that for dimension greater than $2$
such transfer function representations of the corresponding polynomial masks
do not always exist. However, for wavelet masks derived from multivariate polynomials with
 non-negative coefficients,
we determine explicit transfer function representations.
We illustrate our results with several examples.
\end{abstract}


\noindent {\bf Keywords:} multivariate wavelet frame, positive polynomial, sum of hermitian squares, transfer function.\\

\noindent{\bf Math. Sci. Classification 2000:} 65T60, 14P99 11E25,
90C26, 90C22.

\section{Introduction}

A tight wavelet frame of $L_2(\R^d)$ is determined, via Fourier transform, by a
finite set of trigonometric polynomials $p,a_1,\ldots,a_N$.
The trigonometric
polynomial $p$ enters as the unique ingredient into the multiplicative identity
\begin{equation}\label{eq:twoscale}
    \hat\phi(M^T\theta)=p(z)\hat \phi(\theta),\qquad \theta\in\R^d,~~z_j=e^{i\theta_j},
\end{equation}
where $M$ is a $d \times d$ matrix with integer entries whose eigenvalues
are greater than $1$ in absolute value.  The identity \eqref{eq:twoscale} is
called the two-scale relation, as it defines a representation of $\phi$
in terms of shifts of scaled versions of $\phi$, i.e.
\begin{equation}\label{eq:twoscaleb}
  \phi(x)=|\det M| \sum_{\alpha\in\Z^d} p(\alpha)\phi(Mx-\alpha),\qquad
        x\in\R^d.
\end{equation}
Here, $p(z)=\sum_{\alpha\in \Z^d} p(\alpha) z^\alpha$ has finitely many nonzero
 coefficients $p(\alpha)$ and $z^\alpha=z_1^{\alpha_1}\cdots z_d^{\alpha_d}$.

The translation group
$G=2\pi M^{-T}\Z^d / 2 \pi \Z^d$
plays a central role in the discussion of the two-scale relation.
Clearly, $G$ is a finite group of order
 $m = |\det M|$.
Throughout this article we maintain the notation and terminology introduced
in \cite{CPSS2012}.
Our main object of study, as in the previous article \cite{CPSS2012},
is the {\em mask} $p$,
regarded as a Laurent polynomial or, equivalently,
a trigonometric polynomial on the $d$-dimensional torus
$$
  \T^d=\{z=(z_1,\ldots,z_d)\in\C^d: |z_j|=1~\hbox{for}~j=1,\ldots,d\}.
$$
An element of the group  $\sigma=(\sigma_1,
\dots,\sigma_d)\in G$ acts on $p\in\C[\T^d]$ by
$$
 p^\sigma(z)\>:=\>p(e^{-i\sigma_1}z_1,\ldots,e^{-i\sigma_d}z_d),\qquad z\in\T^d.
$$
The conditions
\begin{equation}\label{eq:patone}
   p^\sigma(1,1,\ldots,1)=\delta_{0,\sigma},\qquad \sigma\in G,
    \end{equation}
 are called zero conditions or
sum rules of order $1$ in the literature, see
\cite{JiaJiang2002} and references therein, and are important for the
analysis of various properties of $\phi$.
Another important ingredient of the analysis is the fact that
the support of $\phi$ is contained
in the convex hull of $\{\alpha\in\Z^d: p(\alpha) \ne 0\}$.

We let $F_p=(p^\sigma)_{\sigma\in G}$,
$F_{a_j}=(a_j^\sigma)_{\sigma\in G}:\T^d\to\C^m$ be column vectors.
Then the identity
\begin{equation}\label{eq:UEP}
   I_m-F_p(z)F_p(z)^* = \sum_{j=1}^N F_{a_j}(z)F_{a_j}(z)^*
\end{equation}
is called the {\em Unitary Extension Principle} (UEP) in the seminal work on
frames and shift-invariant spaces by
Ron and Shen \cite{RS95}. Here, $F_p(z)^*=\overline{F_p(z)}^T$
denotes complex conjugation and transposition. If the identities  \eqref{eq:patone} and
\eqref{eq:UEP} are satisfied, then the functions
\begin{equation}\label{eq:TWF}
  \psi_j(x)=|\det M| \sum_{\alpha\in\Z^d} a_j(\alpha)\phi(Mx-\alpha),\qquad
        x\in\R^d,
\end{equation}
are the generators of a tight wavelet frame; i.e. the family
$$
   X(\Psi)=\{
    m^{j/2} \psi_l(M^j\cdot-k): 1\le l\le N,~j\in\Z,~k\in\Z^d\}
$$
defines a tight frame of $L_2(\R^d)$. Therefore, the UEP is the core
of many constructions of tight wavelet frames, see e.g.
\cite{CS, CH,CHS01,DHRS03,HanMo2005,LS06,RS95, Selesnick2001}.

The constraint
\begin{equation}\label{eq:subqmf}
 f =  1-\sum_{\sigma\in G} p^{\sigma *}p^\sigma\ge 0
\end{equation}
is known in the literature as the sub-QMF condition on the trigonometric
polynomial $p\in \C[\T^d]$. Due to $f=\det(I_m-F_pF_p^*)$, the condition in
\eqref{eq:subqmf} is necessary for the existence of $a_1,\ldots,a_N$ that
satisfy  the UEP identities in \eqref{eq:UEP}.

In the first part of this article we
investigate the convex structure of the set of trigonometric polynomials $p$ subject to the
 restrictions \eqref{eq:patone} and \eqref{eq:subqmf}.
The following certificate of the positivity condition turns out to be of great importance
\begin{equation}\label{eq:fpsos}
 f = 1-\sum_{\sigma\in G} p^{\sigma *}p^\sigma = \sum_{j=1}^L |h_j|^2
\end{equation}
where $h_j$ are $G$-invariant trigonometric polynomials. This certificate is called
{\em sum of hermitian squares} (sos) decomposition of the non-negative
trigonometric polynomial $f$.
It was shown in \cite{CPSS2012} and \cite{LS06}
that the condition \eqref{eq:fpsos} is necessary and sufficient for
the existence of trigonometric polynomials $a_j$ in \eqref{eq:UEP}.
The convex structure of the trigonometric polynomials $p$ subject to the
 restrictions \eqref{eq:patone} and \eqref{eq:fpsos}
is more complicated, as this set is not closed.
However, we prove that the extremal points of the
underlying convex sets coincide.
We also show that the sets of those $p$ satisfying either
\eqref{eq:subqmf} or \eqref{eq:fpsos}  are closed under multiplication.
Consequently, we investigate the number of squares $L$
in \eqref{eq:fpsos} of a product $p = p_1 p_2$ in terms of those of the factors
$p_1$ and $p_2$.
Combined with the construction of tight wavelet frames in \cite{LS06},
we obtain better bounds for the number of tight frame generators for
a large class of trigonometric polynomials $p$, including the masks of
multivariate box splines.

So far, we treated the mask $p$ and  $f = 1-\sum_{\sigma\in G} p^{\sigma *}p^\sigma$
as trigonometric polynomials. Since our analysis is not affected
by multiplication of $p$ by a fixed monomial $e^{i\beta\theta}$,
we can assume that $p(z) = \sum_{\alpha\in\N_0^d} p(\alpha) z^ \alpha$ is a polynomial
in $z_j=e^{i\theta_j}$. Therefore, in Section \ref{sec:system_theory_twf}, we
consider $p$ as a complex analytic polynomial $p\in\C[z]$ and
rephrase the decomposition \eqref{eq:fpsos} as
$$ f(z,\overline{z}) = 1-\sum_{\sigma\in G} p^{\sigma}(z)^*p^\sigma(z) =
\sum_{j=1}^L h_j(z)^*h_j(z) + R(z,\overline{z})$$
where $p, h_j$ are complex analytic polynomials, $z$  belongs to the polydisk
$$\D^d=\{z=(z_1,\ldots,z_d): |z_j|<1,~j=1,\ldots,d\},$$
and the residual part
$R(z,\overline{z})$ vanishes on the torus.

This slight change of perspective brings into focus the complex analytic, vector valued polynomial
$$ F_p(z) = (p^\sigma(z))_{\sigma \in G} \in \C^d, \ \ z \in \D^d,$$
subject, by the maximum principle, to the contractivity condition
$$ F_p(z)^* F_p(z) \leq 1, \ \  z \in \D^d.$$
Analytic functions as above, from the polydisk to the unit ball $F_p: \D^d \longrightarrow B(0,1) \subset \C^{m}$,
were intensively studied for more than a century.
The classical works of Schur, Carath\'eodory, Fej\'er and Nevanlinna
have completely settled the intricate structure of analytic functions
from the disk to the disk or the half-plane.
The rather independent and self-sustaining field of bounded analytic
interpolation in one or several complex variables
deals exclusively with such functions.

About half a century ago, electrical engineers, and then many more applied mathematicians, have discovered that
some bounded analytic functions as our $F_p$, or its modification in \eqref{eq:def_fp}, can be interpreted as transfer functions of multivariate, linear
systems appearing in control theory. The second part of our article contains an introduction aimed at the non-expert
to the realization theory of bounded analytic functions in the polydisk. We give precise references to recent and classical
works and we illustrate the benefits of this new dictionary with examples arising in the construction of tight wavelet frames.

In particular we show in Theorem \ref{th:nonnegative_coeff}
  that polynomials $p \in \C[z]$ with non-negative coefficients and satisfying the  conditions \eqref{eq:patone} and \eqref{eq:subqmf}
are transfer functions of finite dimensional linear systems,
complementing and improving typical wavelet theory
results \cite{LS06}. Moreover, we use the adjunction formula for transfer functions,
Proposition \ref{prop:flip_flop}, in order to devise a new technique for
passing from the sos-decomposition in \eqref{eq:fpsos} to the construction of
$a_1,\ldots,a_N$ in the UEP \eqref{eq:UEP}. In Example \ref{ex:B111a} we show that,
even for the simplest nonseparable mask of the
piecewise linear three-directional box-spline $B_{111}$, the techniques from system theory
improve all known frame constructions. Indeed, we obtain
$5$ trigonometric polynomials $a_1,\ldots,a_5$ of coordinate degree $2$,
which complement
the mask of $B_{111}$.

{\bf Acknowledgement.} The second author is indebted to the Gambrinus Fellowship of
Technische Universit\"at Dortmund for support and hospitality in June 2013. The present work
could not be finished without the generous support of the Institute of Mathematics and
Applications in Singapore, where all authors met in December 2013.

\section{Convexity properties of tight wavelet frames} \label{sec:convexity}

In this section we study the properties of the sets of
trigonometric polynomials satisfying the sub-QMF condition and its
subset of trigonometric polynomials which yield tight wavelet
frames.

Denote by $\T^d$ the $d$-dimensional torus
$$
 \T^d\>=\>\{z=(z_1,\ldots,z_d)\in\C^d\colon\>|z_j|=1~\text{for}~j=1,\ldots,d\}.
$$
The vector space $\C[\T^d]$ of
trigonometric polynomials on $\T^d$ is equipped with the finest
locally convex topology under which all linear functionals are
continuous. A basis of neighborhoods of the origin is defined by
the semi-norms
$$
   |p|_\lambda = |\lambda(p)|,\qquad \lambda\in (\C[\T^d])^*,
   \quad p\in\C[\T^d].
$$

For the dilation matrix $M \in \Z^{d \times d}$ define
 $m:=|\hbox{det}(M)|\not=0$. The translation group
$G=2\pi M^{-T}\Z^d / 2 \pi \Z^d$ acts on $p\in\C[\T^d]$ by
$$
 p^\sigma(z)\>:=\>p(e^{-i\sigma_1}z_1,\dots,e^{-i\sigma_d}z_d),
\qquad z\in\T^d,\quad \sigma\in G.
$$
Let $G':=G^*$ be the character group of $G$ and
$p=\sum_{\chi\in G'}p_\chi$ be the isotypical decomposition of
$p$. For each $\chi\in G'$, we choose $\alpha_\chi\in\Z^d$
such that $\tilde{p}_\chi=z^{\alpha_\chi} p_\chi$  is
$G-$invariant. In signal analysis, $\tilde p_\chi$ is called
polyphase component of $p$. Then we have
\begin{equation} \label{def:p_chi,tilde_p_chi}
 \sum_{\sigma\in G} p^{\sigma*} p^\sigma\>=\>
 m \sum_{\chi\in G'} p_{\chi}^* p_\chi=m \sum_{\chi\in G'} \tilde{p}_{\chi}^*
 \tilde{p}_\chi\qquad  \text{on}~\T^d.
\end{equation}

It is well-known that the set of all real-valued, non-negative trigonometric
polynomials is a closed convex cone which is also
closed under multiplication.
We show that the set
$$
   \mathcal{S} =\mathcal{S}(M,d)=
     \{p\in \C[\T^d] : ~p(\bone)=1,~
   f=1-\sum_{\sigma\in G} p^{\sigma *}p^\sigma\ge 0 \ \hbox{on} \ \T^d\}
$$
of trigonometric polynomials that satisfy the sub-QMF condition has the same properties.
Its extremal points are those $p \in \C[\T^d]$ with
$f=0$ on $\T^d$.
Clearly, the set $\mathcal{S}$ is not compact.

Our main interest lies in the set
$$
   \mathcal{S}_{\rm sos} = \{p\in \mathcal{S} : ~
   f=1-\sum_{\sigma\in G} p^{\sigma*}p^\sigma~\mbox{is hermitean sum of squares on} \
   \T^d\},
$$
which is the subset of $\mathcal{S}$ of those $p$ that allow the
sos-representation \eqref{eq:fpsos} of $f$.
By \cite[Theorem 2.2]{CPSS2012} or \cite[Theorem 3.4]{LS06}, the trigonometric polynomials $p\in S_\sos$
yield tight wavelet frames in \eqref{eq:TWF}.
  We show that $\mathcal{S}_{\rm sos}$ is convex, is
not closed for $d\ge3$, and has the same extremal points as
$\mathcal{S}$. Moreover, the set $\mathcal{S}_{\rm sos}$ is closed under
multiplication. This latter property allows us to provide upper
 bounds on the number of the frame generators for box-splines of any dimension,
and to improve known upper bounds for special types of bivariate and
trivariate box-splines in \cite{LaiNam,LS06}.

By the Riesz-Fejer lemma ($d=1$) and Scheiderer's result
\cite{Scheiderer2006} ($d=2$), we have $\mathcal{S}_{\rm sos}=
\mathcal{S}$ for $d=1,2$, while the example in \cite{CPSS2012}
shows $\mathcal{S}_{\rm sos}\subsetneq \mathcal{S}$ for $d \ge 3$.
The following result describes the properties of the set
$\mathcal{S}_{\rm sos}$ also for dimensions $d\ge 3$.

\begin{Theorem}\label{prop:convexity} Let
$d\in\N$.
\begin{description} \item[$(i)$] The set $\mathcal{S}$ is closed and convex. Moreover, $\mathcal{S}$ is closed under
multiplication.

\item[$(ii)$] The set $\mathcal{S}_{sos}$ is convex and is closed under
multiplication.
\end{description}
\end{Theorem}

\begin{proof}
Closedness of the set $\mathcal{S}$  is obvious. To show the
convexity of the sets $\mathcal{S}$ and $\mathcal{S}_{sos}$, let
first $p_1,p_2\in\mathcal{S}$, $t\in (0,1)$  and set $q:=(1-t)p_1
+tp_2$, $u:=p_1-p_2$. The identity
\begin{align*}
q^*q+t(1-t)u^*u\ & =\ (1-t)^2p_1^*p_1+t^2p_2^*p_2+t(1-t)(p_1^*p_2+
  p_2^*p_1) \\
& \quad\ +t(1-t)(p_1^*p_1+p_2^*p_2-p_1^*p_2-p_2^*p_1) \\
& =\ (1-t)p_1^*p_1+tp_2^*p_2
\end{align*}
implies
$$
\sum_{\sigma \in G} q^{*\sigma}q^\sigma+t(1-t)\sum_{\sigma\in G}
u^{*\sigma} u^\sigma\ =\ (1-t)\sum_{\sigma \in G}
p_1^{*\sigma}p_1^\sigma+ t\sum_{\sigma \in G}
p_2^{*\sigma}p_2^\sigma,$$ and, thus, we have
\begin{equation} \label{eq:SSsos}
1-\sum_{\sigma \in G}
q^{*\sigma}q^\sigma\>=\>(1-t)\Bigl(1-\sum_{\sigma \in G}
p_1^{*\sigma}p_1^\sigma\Bigr)+t\Bigl(1-\sum_{\sigma \in G}
p_2^{*\sigma} p_2^\sigma\Bigr)+t(1-t)\sum_{\sigma \in G}
u^{*\sigma}u^\sigma.
\end{equation}
Therefore, $q \in \mathcal{S}$. Furthermore, if
$p_1,p_2\in\mathcal{S}_\sos$, then $q \in \mathcal{S}_{sos}$.

Furthermore, we have
\begin{equation}\label{eq:closedmult}
   1-\sum_{\sigma\in G}(p_1p_2)^{\sigma *} (p_1p_2)^{\sigma} =
   1-\sum_{\sigma\in G}p_1^{\sigma *} p_1^{\sigma}
   +\sum_{\sigma\in G}p_1^{\sigma *} p_1^{\sigma}(1-p_2^{\sigma *} p_2^{\sigma})\ge 0.
\end{equation}
This implies that both $\mathcal{S}$ and $\mathcal{S}_{sos}$ are
closed under multiplication.
\end{proof}

We next characterize the extremal points of $\mathcal{S}$ and
$\mathcal{S}_{sos}$.

\begin{Theorem} \label{th:extremal_points_S} Let $d\in \N$.

\begin{description} \item[$(i)$] The trigonometric
polynomial $p\in \mathcal{S}_\sos$ is an extremal point of $\mathcal{S}_\sos$ if and
only if $p$ satisfies the QMF-condition
\begin{equation} \label{eq:QMF}
   \sum_{\sigma\in G} p^{\sigma *}p^\sigma\equiv 1\qquad\mbox{on}\quad \T^d.
\end{equation}
\item[$(ii)$] The extremal points of $\mathcal{S}$ and
$\mathcal{S}_{sos}$ coincide.
\end{description}
\end{Theorem}

\begin{proof}
Proof of $(i)$: On the one hand, let $q\in \mathcal{S}_\sos$ satisfy
the QMF-condition and assume $q=(1-t)p_1+tp_2$ with $p_1,p_2\in
\mathcal{S}_\sos$, $t\in (0,1)$. From \eqref{eq:SSsos} we conclude that
both $p_1$ and $p_2$ satisfy the QMF-condition in \eqref{eq:QMF}, and
$u=p_1-p_2=0$.
Therefore,
 $q=p_1$ is an
extremal point.

On the other hand, let $p\in \mathcal{S}_\sos$ be an extremal point and let
$r=1-\sum_{\sigma\in G} p^{\sigma *}p^\sigma$.
The isotypical components  $p_\chi$ of $p$, $\chi\in G'$,  satisfy
$$
    m\sum_{\chi\in G'} p_\chi^*p_\chi =\sum_{\sigma\in G} p^{\sigma *}p^\sigma\le 1.
$$
We define the trigonometric polynomials $p_+$ and $p_-$
by their isotypical components
\begin{equation}\label{eq:pplus}
   p_{\pm,0} = p_0\pm \frac{r}{4m},\qquad p_{\pm,\chi}=p_\chi,\quad \chi\ne 0.
\end{equation}
Note that, indeed, $p_0\pm \frac{r}{4m}$ defines an isotypical component,
since both $p_0$ and $r$ are $G$-invariant.
We next show that $p_\pm$ belong to $\mathcal{S}_\sos$.
Note that
\begin{eqnarray*}
  1-m \sum_{\chi \in G'} p_{+,\chi}^* p_{+,\chi}&=&
  1-m \sum_{\chi \in G'} p_{\chi}^* p_{\chi}-\frac{r}{2} \hbox{Re}(p_0)-
  \frac{r^2}{16m} \\ &=& r\left( 1-\frac{1}{2}  \hbox{Re}(p_0)- \frac{r}{16m} \right).
\end{eqnarray*}
By definition of $r$, we have $0\le r\le 1$.
Thus,  $(16m)^{-1}r < 2^{-1}$ and $p_0^*p_0\le 1$,
which imply
$$
  1-\frac{1}{2}  \hbox{Re}(p_0)- \frac{r}{16m} \ge \frac{1}{2}-\frac{r}{16m} >0 \quad
  \hbox{on} \quad \T^d.
$$
This strict positivity, by \cite{Schmuedgen1991}, and the assumption that $r$ is sos yield that
$p_+ \in \calS_\sos$. Analogously,  $p_- \in \calS_\sos$. Since $p=
\frac{1}{2}(p_+ +p_-)$  is extremal,
we conclude that $p=p_+=p_-$ and therefore $r=0$. This shows that $p$
satisfies the QMF condition.

The proof that $\calS$ has the same extremal points is similar to the proof of part (i).
\end{proof}

\begin{Remark}
We have seen that the two sets $\calS_\sos\subset\calS$ are both
convex and their respective extremal points are the same, yet the
inclusion is proper for $d>2$ \cite{CPSS2012}. The following remarks
are an attempt to better understand the geometry of $\calS_\sos$ and $\calS$.

According to the Krein-Milman theorem, any compact convex subset of
$\C[\T^d]$ is the closed convex hull of its extremal points. Although
the convex set $\calS$ is closed, it is easy to see that $\calS$ is
not compact.
So there is no reason to expect that $\calS$ agrees with the closed
convex hull of its extremal points.

For $d>2$, the convex subset $\calS_\sos$ of $\calS$ fails to be
closed. Indeed, let $\calS'$ be the subset of $\calS$ consisting
of all $p\in\C[\T^d]$ for which $p(\bone)=1$ and
$$
 f(z)\>=\>1-\sum_{\sigma\in G}p^{\sigma*}(z)\,p^\sigma(z)\>>\>0\qquad
,z\in\T^d\setminus G,
$$
and for which the Hessian of $f$ at
$\bone$ is positive definite. Then $\calS'\subset\calS_\sos$ by
\cite[Theorem 3.2]{CPSS2012}. But the closure of $\calS'$ is not
contained in $\calS_\sos$. To see this, one can modify the
construction of \cite[Theorem 2.5]{CPSS2012}: Let
$$p_t(z)\>=\>\Bigl(1-t(y_1^2+y_2^2+y_3^2)-c\cdot m(z)\Bigr)\,a(z)$$
where $c$ and $t$ are small positive real numbers and $y_j$, $m(z)$, $a(z)$
are as in \cite{CPSS2012}. When $t$ is small and positive, $p_t$
lies in $\calS'$. But for $t=0$ we have $p_t\notin\calS_\sos$.
\end{Remark}

\section{Bounds on the number of frame generators for box splines}
We use the closedness under multiplication of $\calS_\sos$ to
improve the upper bound in \cite{LaiNam,LS06} for the number $N$ of the
frame generators for box splines. The explicit upper bound for
$N$, see \cite{CPSS2012, LS06}, depends on the length of the sos
decomposition of $f$ and on $m$.

Let $p \in \C[\T^d]$. Denote by $L(p)$ the hermitian sos length of
$\displaystyle f=1-\sum_{\sigma \in G}p^{\sigma *} p^\sigma$,
i.e., the smallest number $r$ such that $f=|h_1|^2+\cdots+|h_r|^2$
with $G$-invariant trigonometric polynomials $h_j\in\C[\T^d]$.  Also, let
$\ell(p)$ be the sos length of $1-p^*p$. As usual we set these
numbers equal to $\infty$, if the respective polynomials are not
sos. Note that $L(p)<\infty$ implies $\ell(p)\le L(p)+m-1<\infty$.

We first prove two auxiliary lemmas.

\begin{Lemma} \label{lem:L(pq)} Let $p,q \in \mathcal{S}$. Then
$$
 L(pq)\le L(p)+ m \ell(q).
$$
\end{Lemma}
\begin{proof} Note that we only need to prove the claim
in the case $L(p)< \infty$ and $\ell(q) < \infty$.  Then,
trivially,
$$
 1-q^*q=\sum_{j=1}^{\ell(q)} \tau_j^* \tau_j \quad \hbox{on} \quad
 \T^d
$$
and also, for $\sigma \in G$,
$$
 1-q^{\sigma *}q^\sigma=\sum_{j=1}^{\ell(q)} \tau_j^{\sigma *} \tau_j^\sigma
 \quad \hbox{on} \quad \T^d.
$$
Note that
$$
 \sum_{\sigma \in G} p^{\sigma *}p^\sigma (1- q^{\sigma *}
 q^\sigma)=\sum_{j=1}^{\ell(q)} \sum_{\sigma \in G} (p\tau_j)^{\sigma *}
 (p\tau_j)^\sigma = m \sum_{j=1}^{\ell(q)} \sum_{\chi \in G'} (\widetilde{p\tau}_j)_\chi^{*}
 (\widetilde{p\tau}_j)_\chi
$$
has a $G-$invariant sos of length $m \ell(q)$. Thus, the claim
follows from
\begin{equation} \label{decomp:Boxsplines}
 1-\sum_{\sigma \in G}(pq)^{\sigma *} (pq)^\sigma=\>\Bigl(1-\sum_{\sigma \in G}
p^{\sigma *} p^\sigma\Bigr)+\sum_{\sigma \in G} p^{\sigma
*}p^\sigma (1- q^{\sigma *}
 q^\sigma).
\end{equation}
\end{proof}

\begin{Lemma} \label{lemma:productG}
Assume that $G_1,\dots,G_r$ are subgroups of $G$ such that the
product map $G_1 \times \cdots \times G_r\to G$ is bijective.
Assume further that for every $j=1,\dots,r$, a polynomial
$p_j\in\C[\T^d]$ is $G_k$-invariant for all $k\ne j$
and is such that $1-\sum_{\sigma\in
G_j} |p_j^{\sigma}|^2$ is sos. Then $p:=p_1 \cdots p_r \in
\mathcal{S}_{\rm sos}$.
\end{Lemma}
\begin{proof} By assumption we have
$p^\sigma=p_1^{\sigma_1}\cdots p_r^{\sigma_r}$ whenever
$\sigma_j\in G_j$, $j=1,\dots,r$, and $\sigma=\sigma_1\cdots\sigma_r$.
Therefore,
$$
 1-\sum_{\sigma\in G}|p^\sigma|^2\>=\> 1-\sum_{\sigma_1\in G_1}
 \cdots\sum_{\sigma_r\in G_r}\bigl|p_1^{\sigma_1}\cdots
 p_r^{\sigma_r} \bigr|^2\>=\>1-\prod_{j=1}^r\Bigl(\sum_{\sigma_j\in
 G_j} |p_j^{\sigma_j}|^2\Bigr).
$$
Denoting $t_j:=\sum_{\sigma_j \in G_j}|p_j^{\sigma_j}|^2$, $j=1,
\dots,r$, we get that
\begin{equation}\label{eq:sos_with_tj}
 1-\sum_{\sigma\in G}|p^\sigma|^2\>= 1-t_1\cdots t_r\>=\>\sum_{j=1}^r
 t_1\cdots t_{j-1}(1-t_j)
\end{equation}
is sos.
\end{proof}

We are finally ready to derive the upper bound for the number of
the frame generators for box splines. From now on  we assume that
$M=2I$, hence $G \cong \pi \{0,1\}^d$.

\begin{Proposition} \label{prop:BOX} Let
$$
 p\>=\>\prod_{j=1}^r \Bigl(\frac{1+z^{\theta_j}}2\Bigr)^{\ell_j},
 \quad \theta_j \in \Z^d, \quad \ell_j \in \N,
$$
and assume that $\theta_1,\dots,\theta_d$ span $\Z^d$ modulo $2\Z^d$.
Then
$$
 L(p) \le d+(r-d)2^d.
$$
\end{Proposition}
\begin{proof}
Define polynomials
$$
 p_{j}(z)=\Bigl(\frac{1+z^{\theta_{j}}}2\Bigr)^{\ell_{j}}, \quad
 j=1,\ldots, r,
$$
which we each treat as a univariate polynomial in the variable
$u_j:=z^{\theta_j}$, respectively. We first show that $L(p_1
\cdots p_d)=d$. For $\theta_j \in \Z^d$ define
$\ol\theta_j=\theta_j+2\Z^d\in \Z^d/2\Z^d$. By assumption,
$\ol\theta_{1},\dots,\ol\theta_{d}$ is a basis of $\Z^d/2\Z^d$.
Let $b_1,\dots,b_d\in\Z^d$ such that $\ol b_1,\dots, \ol b_d$ is
the dual basis of $\Z^d/2\Z^d$. For $j=1,\dots,d$ let $G_j \subset
G$ be the subgroup of order two generated by $\pi b_j$. Then the
group $G$ is the direct product of $G_1,\dots,G_d$. Moreover, for
$j \ne k$ and $\sigma \in G_k$, we have ${\bf e}^{i\sigma \cdot
\theta_{j}} =1$ and, thus, the polynomial $p_j$ is invariant under
$G_k$. Note next that the non-negative polynomials
$$
 t_{j}:=\sum_{\sigma \in G_j} |p_{j}^\sigma|^2=\left|\frac{1+z^{\theta_j}}{2} \right|^{2\ell_j}+
 \left|\frac{1-z^{\theta_j}}{2} \right|^{2\ell_j}, \quad j=1,\ldots,d,
$$
and  $1-t_{j}$ are $G-$invariant.  By the F\'ejer-Riesz
Lemma, $t_{j}$ and $1-t_{j}$, $j=1, \ldots,d$, are therefore single $G-$invariant
squares in the variable $u_j$. Therefore, Lemma
\ref{lemma:productG} implies that $q=p_1 \cdots p_d \in
\mathcal{S}_{sos}$ and, by \eqref{eq:sos_with_tj},  $L(q) = d$.

Next, note that the Riesz-Fejer Lemma implies that $\ell(p_j)=1$,
$j=d+1, \ldots, r$. Thus, by Lemma \ref{lem:L(pq)}, in particular
by identity \eqref{decomp:Boxsplines}, we get $L(qp_{d+1}) \le
L(q)+2^d$, where $|G|=2^d$. The claim follows then by induction on
$n$ for the polynomials $q  \displaystyle \prod_{j=d+1}^n
p_j$, $n=d+2, \ldots, r$.
\end{proof}

\begin{Remark} By the constructive algorithm in \cite{LS06} and by Proposition
\ref{prop:BOX}, for $d=2$ and $r=2,3,4,5\dots$, we get the upper
bounds $6,10,14,18 \dots$ for  the number of tight frame
generators for the corresponding $r-$directional box splines. This
improves the previously known upper bounds from \cite{LaiNam}, namely
$11,19$, for $d=2$ and $r=3,4$. Note that our upper bounds are
 not sharp, in general. For example, for
$$
 p(z_1,z_2)\>=\> \Bigl(\frac{1+z_1}2\Bigr)\Bigl(\frac{1+z_2}2\Bigr)
 \Bigl(\frac{1+z_1z_2}2\Bigr),
$$
i.e. $d=2$ and $r=3$, a tight wavelet frame with only
$6$ frame generators was constructed in \cite{LS06}, and in subsection \ref{subsec:B111}
we construct  a tight wavelet frame with only
$5$ frame generators.
\end{Remark}

\medskip

\section{System theory and wavelet tight frames} \label{sec:system_theory_twf}

In this section we establish a connection between constructions of
tight wavelet frames and some fundamental results from system
theory. For the reader's convenience, we include
 an overview of the
relevant results from system theory in   section \ref{appendix}.

Here, instead of working with trigonometric polynomials, we
consider algebraic polynomials $p \in \C[z]$. We write $M=(m_1,\ldots,m_d)\in\Z^{d\times d}$
and define the isotypical components $p_\chi$ and
the polyphase
components $\tilde p_\chi$, $\chi \in G'$,  similarly to
\eqref{def:p_chi,tilde_p_chi}. Hence,  the polyphase components
of $\displaystyle p=\sum_{\beta \in \Z^d}  p(\beta) z^\beta$  are
\[ \tilde p_\chi=z^{-\alpha_\chi}p_\chi =
\sum_{\beta\in\Z^d} p(\alpha_\chi+M\beta) z^{M\beta}.
\]
Therefore, we consider $\tilde p_\chi$ as
polynomials in the variable $\xi=z^M:=(z^{m_1},\ldots,z^{m_d}) $,
due to the identity $\xi^\beta=z^{M\beta}$.

Define the vector-valued analytic function $f_p: \C^d \rightarrow
\C^m$ by
\begin{equation}\label{eq:def_fp}
 f_p(\xi)=(m^{1/2} \tilde{p}_\chi(\xi))_{\chi \in G'}.
\end{equation}
Then, in the polarized version with variables $\xi,\eta\in\C^d$, we have
$$
   1-m\sum_{\chi \in G'}
\tilde{p}_{\chi}(\eta)^* \tilde{p}_\chi(\xi)= 1- f_p(\eta)^*
f_p(\xi).
$$
We assume that the analytic function $f_p$ satisfies $\|f_p(\xi)\|
\le 1$ for all $\xi$ in the polydisk
$$
 \D^d=\left\{ \xi=(\xi_1,\ldots,\xi_d) \in \C^d \ : \
|\xi_j|< 1, \ j=1,\ldots,d \right\}.
$$
If $\|f_p(\xi)\|=1$ on $\T^d$, then $f_p$ is called \emph{inner}.
The   requirement that either $\|f_p(\xi)\|\le 1$ or $\|f_p(\xi)\|=1$
states that the trigonometric polynomial
$p|_{\T^d}$ satisfies either the sub-QMF \eqref{eq:subqmf} or QMF \eqref{eq:QMF}
condition,
respectively.

It is then natural to ask, if such functions $f_p$ possess the
decomposition
\begin{equation} \label{eq:Schur_Agler_representation:inner}
 1-f_p(\eta)^* f_p(\xi)=
 q_0(\eta)^*q_0( \xi)+\sum_{j=1}^d (1-\xi_j \bar{\eta_j})q_j(\eta)^* q_j(\xi)
\end{equation}
with polynomial maps $q_j: \C^d \rightarrow \C^{N_{j}}$, $N_{j}
\in \N$ and $j=0,\ldots,d$. If we consider $\xi=\eta\in\T^d$ in
\eqref{eq:Schur_Agler_representation:inner}, then the last sum
disappears, and with $q_0=(h_1,\ldots,h_{N_0})^T$, we obtain the sos decomposition
in \eqref{eq:fpsos} with $G$-invariant trigonometric polynomials $h_1,\ldots,h_{N_0}$,
$$
  1-\sum_{\sigma\in G}p^{\sigma *}p^\sigma=
    1-f_p(\xi)^*f_p(\xi)= \sum_{j=1}^{N_0} h_j(\xi)^*h_j(\xi), \quad \xi=z^M.
$$
In other words, by construction of
the decomposition \eqref{eq:Schur_Agler_representation:inner} on $\D^d$, we prove that
the trigonometric polynomial $p|_{\T^d}$ is in $\mathcal{S}_\sos$ and, in addition,
 we have the sos-decomposition \eqref{eq:fpsos} of sos-length $N_0$.
Moreover, we connect the
bilinear decomposition \eqref{eq:Schur_Agler_representation:inner}
 with the
realization formula
$$
   \begin{pmatrix}f_p\\q_0\end{pmatrix}(\xi)=
    A+BE(\xi)(I-DE(\xi))^{-1} C,\qquad \xi\in \D^d,
$$
in Theorem \ref{th:Agler_Ball_Trent}(c) and
obtain a parameterized version (in terms of the isometry $\begin{pmatrix}A&B\\C&D\end{pmatrix}$) of the decomposition \eqref{eq:fpsos}.

This motivates us to study the properties of the set
$$
 \mathcal{S}_{A}:=\{p \in \C[z] \ : \
  f_p  \
 \hbox{satisfies \eqref{eq:Schur_Agler_representation:inner}}\}.
$$
Note that, for $p\in \mathcal{S}_{A}$, the function $(f_p,q_0)^T :
\C^d\to \C^{m+N_0}$ is inner and in the Schur-Agler class, see
Theorem \ref{th:Agler_Ball_Trent} with $X=Y=\C$.

Unfortunately, if the corresponding trigonometric
polynomial $p|_{\T^d}$ is in $\mathcal{S}_{sos}$, then we do not
necessarily have $p \in \mathcal{S}_A$. The following example
illustrates this observation.

\begin{Example}\label{ex:druryb}
Let $g(z)=z_1^3+z_2^3+z_3^3-3z_1z_2z_3$,
$z \in \D^3$, be the polynomial in \eqref{Druryex}. The maxima of
$|g|$ are at
$$
z_1=z_2=e^{2\pi i /3}z_3,
\qquad z_3\in \T,
$$
and
$$ z_1=z_2=e^{-2\pi i /3}z_3,
\qquad  z_3\in \T,
$$
and permutations thereof.
We select
$$
 z_1=z_2=e^{2\pi i /3},\quad z_3=1,  
$$
where $|g(z_1,z_2,z_3)|=\|g\|_{\infty,\bar \D^3}=3 \sqrt{3}$.
We define another polynomial
$$
 q(z)=\frac{g(e^{2\pi i /3}z_1,e^{2\pi i /3}z_2,z_3)}{g(e^{2\pi i /3},e^{2\pi i /3},1)}
=\frac{1}{3(1+e^{\pi i/3})}
\left(z_1^3+z_2^3+z_3^3+3e^{\pi i /3}z_1z_2z_3\right)
$$
so that $q(\bone)=1$, $\|q\|_{\infty,\bar\D^3}=1$ and
$\|q(T_1,T_2,T_3)\|=\frac{2}{\sqrt{3}}$ with the appropriately
rotated commutative contractions $T_1$, $T_2$ and $T_3$ for which
$\|g(T_1,T_2,T_3)\|=6$, see subsection
\ref{subsec:Appendix_several_variables}. Next, for the dilation
matrix $M =2I$, we define
$$
 \xi=(\xi_1,\xi_2,\xi_3)=(z_1^2, z_2^2, z_3^2) \in \D^3
$$
and the polynomial $p \in \C[z]$  by
$$
 p(z)= 8^{-1} q(z^2) \sum_{\chi \in G'}  z^{\alpha_\chi}, \quad z \in
 \D^3,\quad
 \alpha_\chi \in \Gamma=\{0,1\}^3.
$$
The corresponding column vector $f_p: \C^3 \rightarrow \C^8$ of
the polyphase components of $p$ is given by
$$
 f_p(\xi)=8^{-1/2}q(\xi)\left( \begin{array}{c} 1 \\ \vdots \\1 \end{array} \right).
$$
The polynomial $p$ satisfies $p(\bone)=1$, and
$$
 1-f_p(\eta)^*f_p(\xi)=1-q(\eta)^*q(\xi)
$$
does not possess the representation in
\eqref{eq:Schur_Agler_representation:inner}, i.e. $p \not \in
\mathcal{S}_A$. We show next that $p|_{\T^3} \in
\mathcal{S}_{sos}$. Note that if we dehomogenize
$$
 q(z^2)=\frac{z_3^6}{3(1+e^{\pi i/3}}
 \left(\frac{z_1^6}{z_3^6}+\frac{z_2^6}{z_3^6}+1+
 3e^{\pi i/3}\frac{z_1^2z_2^2}{z_3^4}\right),
 \quad z \in \T^3,
$$
and set
$$
 y_1=\frac{z_1^2}{z_3^2} \quad \hbox{and} \quad y_2=\frac{z_2^2}{z_3^2},
$$
then  the polynomial $1-q^*q$ in the variables $y_1$ and $y_2$ is
a $2-$dim non-negative polynomial on $\T^2$. Thus, by
\cite{Scheiderer2006}, $p|_{\T^3} \in \mathcal{S}_{sos}$.
\end{Example}

\subsection{Polynomials $p$ with non-negative coefficients}

Despite the difficulties illustrated in example \ref{ex:druryb}, we are able to
describe a large class of analytic polynomials $p$ which belong to the
 set $\mathcal{S}_A$.

\begin{Theorem} \label{th:nonnegative_coeff} Let $p \in \C[z]$ have non-negative
coefficients, and let its polyphase components satisfy
$\widetilde{p}_\chi(\bone)=m^{-1}$, $\chi \in G'$. Then $p \in \mathcal{S}_A$.
\end{Theorem}
\begin{proof}
Let $\Gamma \subset \N_0^d$ be a set of representatives of $G'$.
Define the sets
$$
 \cI=\{\alpha \in \Z^d \ : \  p(\alpha) \not=0\} \quad \hbox{and}
 \quad \tilde{\cI}=\{\alpha \in \Z^d \ : \ \exists \gamma \in \Gamma \ \hbox{such that}
 \ \gamma+M\alpha \in \cI\}.
$$
Assume that the index set $\tilde{\cI}$ is linearly ordered, e.g.
by the lexicographical ordering. Also, for $j=1, \ldots, d$, we
define
$$
 n_j=\max\{\alpha_j \ : \  \alpha=(\alpha_1, \ldots, \alpha_d) \in \tilde{\cI}\}
$$
and
\begin{equation} \label{def:I_j}
   \cI_j=\{0,\ldots,n_1\} \times   \ldots \times \{0,\ldots,n_j\}.
\end{equation}
Furthermore, we define the row and column vectors
$$
 \pb_\chi=(p(\alpha_\chi+M\alpha) \ : \ \alpha \in \tilde{\cI}), \quad \chi \in G', \quad
 \hbox{and} \quad v(\xi)=(\xi^\alpha\ :\ \alpha \in \tilde{\cI})^T,
$$
respectively. Define also the monomial vector
\begin{equation}\label{def:v_j}
 v_j(\xi_1, \ldots, \xi_d)=
 (\xi_1^{\beta_1} \cdots \xi_j^{\beta_j}\ :\  \beta=(\beta_1, \ldots, \beta_j) \in  \cI_j)^T.
\end{equation}
Note that $v(\eta)^*=(\bar{\eta}^\alpha \ : \ \alpha \in
\tilde{\cI})$ and $v_j(\eta)^*=(\bar{\eta}^\alpha \ : \ \alpha
\in \cI_j)$ are then  row vectors.

Next we write the polyphase components of $p$ in the vector form
$$
 \tilde{p}_\chi(\xi)=\sum_{\alpha \in \tilde{\cI}} p(\alpha_\chi+M\alpha)
 \xi^{\alpha}=\pb_\chi \cdot v(\xi), \quad \xi \in \C^d.
$$
By Theorem \ref{th:Agler_Ball_Trent}, it suffices to show that
\begin{equation}\label{eq:Aj}
 1 - m\sum_{\chi \in G'} v(\eta)^* \pb_\chi^* \pb_\chi v(\xi) =
 v(\eta)^*A_0 v(\xi)+ \sum_{j=1}^d (1-\xi_j \bar{\eta}_j) v_j(\eta)^* A_j v_j(\xi)
\end{equation}
for $\xi, \eta \in \D^d$ with hermitean positive semi-definite matrices
$A_j$, $j=0, \ldots,d$. Due to $\widetilde{p}_\chi(\bone)=m^{-1}$, we have
$$
 1=\displaystyle \sum_{\chi \in G'} \sum_{\alpha \in \tilde{\cI}}
 p(\alpha_\chi+M\alpha).
$$
Thus, we get
\begin{eqnarray*}
 && 1-m\sum_{\chi\in G'} v(\eta)^* \pb_\chi^* \pb_\chi v(\xi)=
 1-m\sum_{\chi \in G'} \sum_{\alpha,\beta\in \tilde{\cI}} p(\alpha_\chi+M\alpha)
 p(\alpha_\chi+M\beta) \xi^\alpha\overline{\eta}^\beta
 \\
  && = \sum_{\chi \in G'} \Big(
 \sum_{\alpha\in \widetilde{\cI}}p(\alpha_\chi+M\alpha) -m\sum_{\alpha\in \widetilde{\cI}}
 p(\alpha_\chi+M\alpha)^2 \xi^\alpha\overline{\eta }^\alpha \\
 && \hspace{4.5cm} -m  \sum_{\alpha,\beta\in\widetilde{\cI}\atop \alpha \not= \beta}
 p(\alpha_\chi+M\alpha) p(\alpha_\chi+M\beta)
 \xi^\alpha\overline{\eta}^\beta \Big)\\
&& = \sum_{\chi \in G'} \Big(
 \sum_{\alpha\in \widetilde{\cI}}(p(\alpha_\chi+M\alpha) -mp (\alpha_\chi+M\alpha)^2)
 \xi^\alpha\overline{\eta}^\alpha\\
 && \hspace{1cm} -m  \sum_{\alpha,\beta\in\widetilde{\cI}\atop \alpha \not= \beta}
 p(\alpha_\chi+M\alpha)
 p(\alpha_\chi+M\beta)\xi^\alpha\overline{\eta}^\beta
 + \sum_{\alpha\in\widetilde{\cI}} (1- \xi^\alpha\overline{\eta}^\alpha) p(\alpha_\chi+M\alpha) \Big).
\end{eqnarray*}
Define  the $|\widetilde{\cI}| \times |\widetilde{\cI}|$  matrices
$A_{\chi,0}$, $\chi \in G'$, by
$$
 A_{\chi,0}(\alpha,\beta)=
  \begin{cases}
    \displaystyle  p(\alpha_\chi+M\alpha) -mp (\alpha_\chi+M\alpha)^2, & \text{if} \quad  \alpha=\beta, \\
    -m p(\alpha_\chi+M\alpha) p(\alpha_\chi+M\beta), & \text{otherwise},
  \end{cases} \quad \alpha,\beta \in \tilde{\cal I}.
$$
The simple observation
\begin{eqnarray*}
 p(\alpha_\chi+M\alpha)= p(\alpha_\chi+M\alpha) \, m \, \widetilde{p}_\chi(\bone)=
 m \, p(\alpha_\chi+M\alpha) \sum_{\beta \in \tilde{\cal I}} p(\alpha_\chi+M\beta)
\end{eqnarray*}
implies that $A_{\chi,0}$, $\chi \in G'$, are weakly diagonally dominant
and, thus, are positive semi-definite. Therefore,
$$
   1-m\sum_{\chi \in G'} v(\eta)^* \pb_\chi^* \pb_\chi v(\xi)=
   v(\eta)^* A_0 v(\xi) +
   \sum_{\alpha\in \widetilde{\cI}} (1- \xi^\alpha\overline{\eta}^\alpha)
   \sum_{\chi \in G'} p(\alpha_\chi+M\alpha)
$$
with the positive semi-definite matrix $A_0=\displaystyle
\sum_{\chi \in G'} A_{\chi,0}$.
For $j=1, \ldots, d$ and $\beta \in \cI_j$, we also define
\begin{equation}\label{eq:Ibeta}
  \cI(\beta)=\{\alpha\in \tilde{\cI} \ : \  \alpha_k=\beta_k~\textrm{for}~1\le k\le j-1,~~\alpha_j>\beta_j\}.
\end{equation}
Then, due to
$$
   1-\eta^\alpha=1-\eta_1^{\alpha_1}+\eta_1^{\alpha_1}(1-\eta_2^{\alpha_2})+\cdots
   + \eta_1^{\alpha_1}\cdots
   \eta_{d-1}^{\alpha_{d-1}}(1-\eta_d^{\alpha_d}), \  \eta\in\C^d, \ \alpha \in
   \N_0^d,
$$
and
$$
 1-\eta_j^{\alpha_j}=(1-\eta_j)\sum_{k=0}^{\alpha_j-1}\eta_j^k, \quad
 \alpha_j>0,
$$
we obtain
\begin{eqnarray*}
   &&\sum_{\alpha\in \widetilde{\cI}} (1- \xi^\alpha\overline{\eta}^\alpha)  \sum_{\chi \in G'} p(\alpha_\chi+M\alpha)
    = \\  && \hspace{1cm} \sum_{j=1}^d(1-\xi_j\bar{\eta}_j) \sum_{\beta \in \cI_j} (\xi_1,\ldots,\xi_j)^\beta
    (\overline{\eta}_1,\ldots,\overline{\eta}_j)^\beta \sum_{\alpha \in \cI(\beta)} \sum_{\chi \in G'}
     p(\alpha_\chi+M\alpha).
\end{eqnarray*}
Then, we get
\begin{equation} \label{def:A_j}
\sum_{\alpha\in \widetilde{\cI}} (1- \xi^\alpha\overline{\eta}^\alpha)
\sum_{\chi \in G'}
     p(\alpha_\chi+M\alpha)=
  \sum_{j=1}^d(1-\xi_j\bar{\eta}_j) v_j(\eta)^* A_{j} v_j(\xi),
\end{equation}
with diagonal $| \cI_j| \times |\cI_j|$ matrices
$A_j$ whose non-negative diagonal entries are $A_j(\beta,
\beta)=\displaystyle \sum_{\alpha\in \cI(\beta)} \sum_{\chi \in G'} p(\alpha_\chi+M\alpha)$, $\beta
\in \cI_j$.
\end{proof}

\begin{Remark} \label{rem:L0_Lj} Note that the matrices $A_0$ and $A_j$, $j=1, \ldots, d$, in \eqref{def:A_j} define the
polynomial maps $q_j$, $j=0, \ldots, d$, in
\eqref{eq:Schur_Agler_representation:inner} by
\begin{equation} \label{eq:q0xi}
 q_0(\xi)=\sqrt{A_0} v(\xi) \quad \hbox{and} \quad q_j(\xi)=\sqrt{A_j} v_j(\xi).
\end{equation}
Since $A_j$ is diagonal and the entries of $v_j$ belong to $\C[\xi_1,\ldots,\xi_j]$, the vector
$q_j(\xi)$ is a vector of (scaled) monomials in $\C[\xi_1,
\ldots, \xi_j]$, $j=1, \ldots,d$.
\end{Remark}

Theorem \ref{th:nonnegative_coeff} and  Theorem
\ref{th:Agler_Ball_Trent} imply that the maps $f_p$ and  $q_0$
satisfy
\begin{equation} \label{representation_of_F_G_q_0}
 \left(\begin{array}{c} f_p \\ q_0 \end{array} \right)(\xi)
 =A+BE(\xi)(I-DE(\xi))^{-1} C, \quad \xi\in \D^d,
\end{equation}
with an isometry
$$ \left( \begin{array}{cc} A & B\\
                            C & D\\
           \end{array} \right) : \begin{array}{c} \C\\
           \oplus\\
           \C^{|\cI_1|+\ldots+|\cI_d|}\end{array} \longrightarrow \begin{array}{c} \C^{m+r} \\
           \oplus \\
           \C^{|\cI_1|+\ldots+|\cI_d|} \end{array}, \quad
           r=|\tilde{\cI}|,
$$
and the block diagonal matrix
\begin{equation} \label{def:Ez}
 E(\xi)=\hbox{diag}(I_{|\cI_1|}\xi_1, \ldots, I_{|\cI_d|}\xi_d).
\end{equation}
To obtain the contractive representation for $f_p$ one just
deletes the last $r$ rows of $A$ and $B$ and leaves $C$ and $D$
unchanged. In the proof of the following Corollary we define one
possible choice of the matrices $A$, $B$, $C$ and $D$ and study
the properties of $D$.

\begin{Corollary} \label{cor:structure_D} Under assumptions of Theorem
\ref{th:nonnegative_coeff}, there exists a realization
$$
   \left(\begin{array}{c}f_p \\ q_0 \end{array} \right)(\xi) = A+B E(\xi)(I-D E(\xi))^{-1}C, \quad \xi
   \in \D^d,
$$
with   nilpotent matrix $DE(\xi)$.
\end{Corollary}
\begin{proof}
Let the matrices $A_j$, $0\le j\le d$, be defined as in the proof
of Theorem \ref{th:nonnegative_coeff}, the sets $\cI_j$ and column
vectors $v_j$ be as in \eqref{def:I_j} and \eqref{def:v_j} (lexicographical ordering),
respectively. Then, by Theorem \ref{th:nonnegative_coeff}, we have
$$
 1-f_p(\eta)^*f_p(\xi)-q_0(\eta)^*q_0(\xi)=
 \sum_{j=1}^d
 (1-\xi_j\bar{\eta}_j) \left(\sqrt{A_j} v_j(\eta)\right)^* \sqrt{A_j}v_j(\xi),
$$
where $q_0(\xi)= \sqrt{A_0}v(\xi)$ is a polynomial map from $\C^d$
into $\C^{r}$, $r=|\tilde{\cI}|$. The existence of the isometry
$\left( \begin{array}{cc} A & B\\C & D\\ \end{array} \right)$ in
\eqref{representation_of_F_G_q_0} is guaranteed by Theorem
\ref{th:Agler_Ball_Trent} part c). Next we explicitly derive one
such possible $ABCD-$representation for $f_p$. Let
$$
 g(\xi)=\left(\begin{array}{c} q_1(\xi) \\ \vdots \\ q_d(\xi) \end{array} \right).
$$
For $E(\xi)$ in \eqref{def:Ez}, from
\begin{equation} \label{ABCD_identity}
 \left( \begin{array}{cc} A & B\\C & D\\ \end{array} \right)
 \left( \begin{array}{c} I \\ E(\xi)g(\xi) \end{array}\right)=
 \left( \begin{array}{c} \left( \begin{array}{c} f_p \\ q_0 \end{array} \right) (\xi) \\ g(\xi) \end{array}\right),
\end{equation}
we immediately get
$$
 A=\left(\begin{array}{c} f_p \\ q_0\end{array} \right)(0) \in
 \C^{m+r} \quad \hbox{and} \quad C=g(0)=\left( \begin{array}{c} q_1(0) \\ \vdots \\ q_d(0)
 \end{array}\right) \in \C^{|\cI_1| + \dots+ |\cI_d|}.
$$
 Note next that to determine $D$ and $B$ in \eqref{ABCD_identity}
we need to solve
\begin{equation} \label{eq:D+B}
 DE(\xi) g(\xi)=g(\xi)-C \quad \hbox{and} \quad  B E(\xi) g(\xi)=\left(\begin{array}{c} f_p \\ q_0\end{array} \right)(\xi)-A.
\end{equation}
We start by determining the entries of the matrix $D$, which we write in the block form
$$
 D=\left(\begin{array}{ccc} D_{11} & \ldots & D_{1d} \\  \vdots &
 & \vdots \\D_{d1} & \ldots & D_{dd} \end{array}\right), \quad D_{ij}
 \in \C^{|\cI_i| \times |\cI_j|}.
$$
We index the entries $D_{ij}(\beta, \gamma)$ in the block $D_{ij}$
according to the lexicographical ordering of $\beta \in \cI_i$ and
$\gamma \in \cI_j$. We first observe that, due to $q_j \in
\C[\xi_1, \ldots, \xi_j]$ (see Remark \ref{rem:L0_Lj}) and by the
first identity in \eqref{eq:D+B}, the matrix $D$ is block lower
triangular. Then, from the first identity in  \eqref{eq:D+B},
for the blocks in the $i$-th row of $D$  we get
$$
 \sum_{j=1}^i \xi_jD_{ij} q_j(\xi)=q_i(\xi)-q_i(0), \quad i=1, \ldots,d,
$$
where, by \eqref{eq:q0xi}, the entries of $q_j(\xi)$ are either
equal to zero or are scaled monomials
$\sqrt{A_j(\beta,\beta)} (\xi_1,\ldots,\xi_j)^\beta$,
$\beta \in \cI_j$. For
each $i=1, \ldots, d$, we proceed as follows. Choose a
non-negative entry in $q_i(\xi)-q_i(0)$. It corresponds to
a non-zero diagonal element $\sqrt{A_i(\beta,\beta)}$ for some $\beta \in \cI_i$.

\begin{description}
 \item[Case 1:] If $\beta=(\beta_1, \ldots, \beta_i)$ with $\beta_i >0$, then
set $j=i$ and $\gamma=(\beta_1, \ldots, \beta_i-1) \in \cI_i$.

\item[Case 2:] If $\beta=(\beta_1, \ldots, \beta_j, 0, \ldots, 0)$ with $j<i$ and
$\beta_j >0$, then
set $\gamma=(\beta_1, \ldots, \beta_j-1) \in \cI_j$.
\end{description}

\noindent By \eqref{eq:Ibeta}, we get that $I(\beta) \subset I(\gamma)$, which implies that $A_j(\gamma,\gamma) \ge A_{i}(\beta,\beta)>0$. Define
$$
 D_{ij}(\beta, \gamma)=\sqrt{\frac{A_i(\beta,\beta)}{A_j(\gamma,\gamma)}}.
$$
Note that, due to the structure of $q_i(\xi)$, the block $D_{ij}$
has at most one non-negative entry $ D_{ij}(\beta, \gamma)$ in each
row. Also, for $i=j$ (Case 1) and
due to $\gamma_i<\beta_i$,
the blocks $D_{ii}$ are lower triangular, with zeros on the main diagonal. This implies
that $DE(\xi)$ is nilpotent.

Similarly, we determine the non-zero elements of the matrix $B$, which we write as a block matrix of the form
$$
 B=\left(\begin{array}{ccc} B_{11}& \dots & B_{1d} \\ B_{21}& \dots & B_{2d}
 \end{array} \right), \quad B_{1j} \in \C^{|G'| \times |\cI_j|},
 \quad B_{2j} \in \C^{|\tilde{\cI}| \times |\cI_j|}.
$$
Recall that the second identity in \eqref{eq:D+B} is of the form
$$
 \left(\begin{array}{c} f_p \\ q_0\end{array} \right)(\xi)-A=
 \left(\begin{array}{c} \left(\displaystyle  m^{1/2}\sum_{\alpha \in \tilde{\cI}}
  p(\alpha_\chi+M \alpha) \xi^\alpha \right)_{\chi \in G'} \\ \sqrt{A_0} v(\xi)\end{array} \right)-A
$$
with $v(\xi)=(\xi^\alpha \ : \ \alpha \in \tilde{\cI} )^T$. By
the same argument as above, for each $\alpha \in \tilde{\cI}$ we
determine $j \in \{1, \ldots, d\}$ and $\gamma \in \cI_j$ such
that $A_j(\gamma, \gamma) > 0$ and $\alpha=(\gamma_1,
\ldots,\gamma_j+1, 0 \ldots, 0)$. Then non-zero entries of
$B_{1,j}$ blocks are defined by
$$
  B_{1j}(\chi,\gamma)=m^{1/2}\frac{ p(\alpha_\chi+M
  \alpha)}{\sqrt{A_j(\gamma,\gamma)}}, \quad \chi \in G'.
$$
Analogously for the blocks $B_{2j}$, $j=1, \ldots,d$.
\end{proof}

\medskip

\subsection{Matrix factorization: univariate case} \label{subsec:example_system}

We consider the univariate case where $M=m\in\N$ and $m\ge 2$. We use the results of
system theory to give an alternative proof of \cite[Theorem
4.1]{HHS} which shows how to construct a tight frame with $m$
generators. With notation in \eqref{eq:def_fp}, this requires us to find
a matrix factorization
\begin{equation}\label{eq:ex1d}
I_m-f_p(\xi)f_p(\xi)^*=U( \xi)U(\xi)^*,\qquad \xi\in\T,
\end{equation}
where $U$ is a polynomial matrix of dimension $m\times m$. Note that
\eqref{eq:ex1d} are the UEP identities for \eqref{eq:UEP}, written in terms of
the polyphase components $f_p$ of $p$ instead of the $G$-shifts $F_p$.
(Passing from the vector $F_p$ to $f_p$ eliminates the dependencies
among the components of $F_p$.)
Then the columns of $U=(u_{\chi,j})_{\chi\in G',\ j=1,\ldots,m}$
define the polyphase components $\tilde a_{j,\chi}=u_{\chi,j}$ of each
trigonometric polynomial $a_j$ in \eqref{eq:UEP}, i.e.
$$
   a_j(z)= \sum_{\chi\in G'} z^{-\alpha_\chi}u_{\chi,j}(\xi),\qquad
    \xi=z^M\in \T.
$$
The following result shows that such a matrix $U$ can be constructed
by the scalar Riesz-Fejer lemma and the adjunction formula
in Proposition \ref{prop:flip_flop}.

\begin{Lemma} Assume that $f:\D\to\C^m$ is a polynomial map with
$\|f(\xi)\| \le 1$ in $\D^1$. Then there exist polynomial
maps $U : \C \rightarrow \C^{m \times m}$ of degree $n=\deg f$
and $k : \C \rightarrow \C^{m\times n}$ of degree less than $n$ such that
$$
 I_m-f(\xi) f(\eta)^*=U(\xi)U(\eta)^*+(1-\xi \bar{\eta})k(\xi)k(\eta)^*,
 \quad \xi,\eta \in \C.
$$
\end{Lemma}
\begin{proof}
Due to $\|f(\xi)\| \le 1$ in $\D^1$ and by the Riesz-Fejer Lemma, there is
a polynomial $q_0\in\C[\xi]$ of degree $n=\deg f$, such that
$$
   \|f(\xi)\|^2 +|q_0(\xi)|^2 \equiv 1,\qquad \xi\in\T^1.
$$
In other words, the polynomial function
$f_1:=\begin{pmatrix}f\\q_0\end{pmatrix}:\C\to\C^{m+1}$
is inner. By Theorem  \ref{th:CW}, there is a  polynomial map
$q_1: \C \rightarrow \C^n$ such that
$$
 1-f(\eta)^* f(\xi)=q_0(\eta)^*q_0(\xi)+(1-\xi \bar{\eta})q_1(\eta)^*q_1(\xi),
 \quad \xi,\eta \in \D^1.
$$
Corollary \ref{cor:general_D_nilpotent} implies that the inner
function $f_1$ possesses the representation
$$f_1(\xi)=\left( \begin{array}{c}A\\A_0\end{array}\right)+\xi
\left( \begin{array}{c}B\\B_0\end{array}\right)(I-\xi D)^{-1}C,$$
where the matrix
$$  \left( \begin{array}{cc} A & B\\
                            A_0 & B_0 \\
                            C & D\\
           \end{array} \right) : \begin{array}{c} \C\\
           \oplus\\
           \C^{n}\end{array} \longrightarrow \begin{array}{c} \C^{m+1} \\
           \oplus \\
           \C^{n} \end{array}
$$
is isometric and $D\in\C^{n\times n}$ is nilpotent.

Using the adjunction formula of Proposition \ref{prop:flip_flop}, we
obtain
$$
f^*(\xi)=
A^*+\xi C^*k^*(\xi),
$$
where $k(\xi)=B(I-\xi D)^{-1}\in \C^{m\times n}$
is also a polynomial map of degree less than $n$.
(Note that $k$ is denoted by $q_1$ in Proposition \ref{prop:flip_flop}.)
Since the matrix
$\left( \begin{array}{cc} A^* & C^*\\ B^*&D^* \end{array} \right)$ is
contractive, we can choose an extension to an isometry
\begin{equation}\label{eq:isoextension}
   \left( \begin{array}{cc} A^* & C^*\\ X&Y\\ B^*&D^* \end{array} \right)
 \end{equation}
in the following way: we first extend the co-isometry
$\left( \begin{array}{ccc} A^* & A_0^*&C^*\\ B^*&B_0^*&D^* \end{array} \right)$ to a unitary matrix
$$\left( \begin{array}{ccc} A^* & A_0^*&C^*\\ X&X_0&Y\\B^*&B_0^*&D^*
\end{array} \right)$$
and then drop the middle columns indexed by $0$. This shows that the isometric extension
\eqref{eq:isoextension} exists with $X\in \C^{m\times m}$, $Y\in \C^{m\times n}$.
We define the polynomial function $U(\xi)\in \C^{m\times m}$ by
$$
   U^*(\xi)= X+ \xi Y k^*(\xi)
$$
and obtain the claim.
\end{proof}

\subsection{Bivariate example: piecewise linear box-spline}
\label{subsec:B111}

The following simple, but educational, example illustrates the
result of Corollary \ref{cor:structure_D} in the bivariate case,
where $p$ is the polynomial associated with the linear three-directional
box-spline. In particular, it shows how to derive
the $ABCD$-representation of $f_p$.

\begin{Example} \label{ex:B111}
 Let $M=2I_2$, $m=4$, and  consider
 $$
  p(z_1,z_2)=\frac{1}{8} \left(1+z_1+z_2+2z_1z_2+z_1z_2^2+z_1^2z_2+z_1^2z_2^2\right).
 $$
 Let $\xi_j=z_j^2$, $j=1,2$ and $v(\xi)=\left( \begin{array}{cccc}1&\xi_1&\xi_2&\xi_1\xi_2 \end{array} \right)^T$.
 Then
 $$
 f_p(\xi)=\left( m^{1/2} \tilde{p}_\chi(\xi) \right)_{\chi \in G'}=
  \frac{1}{4} \left(\begin{array}{c} 1+\xi_1\xi_2 \\1+\xi_2\\1+\xi_1 \\2 \end{array} \right).
 $$
 Using the construction in the proof of Theorem \ref{th:nonnegative_coeff}, we get
 $$
  1-f_p(\eta)^* f_p(\xi)=v(\eta)^* A_0 v(\xi)+\sum_{j=1}^2
 (1-\xi_j\bar{\eta}_j) v_j(\eta)^* A_j v_j(\xi)
 $$
 with $v_1(\xi)=1$, $v_2(\xi)=\left(\begin{array}{cc}1 & \xi_1\end{array} \right)^T$,
 $$
  A_0=\frac{1}{16}\left(\begin{array}{rrrr}3&-1&-1&-1\\-1&1&0&0\\-1&0&1&0\\-1&0&0&1\end{array} \right),
  \quad A_1=\frac{1}{4} \quad \hbox{and} \quad
  A_2=\frac{1}{8}\hbox{diag}(1,1).
 $$
 Note that the positive semi-definite matrix $A_0$ has rank $3$ and admits the factorization
 $$
  A_0=H_0^TH_0 \quad \hbox{with} \quad
  H_0=\frac{1}{4}\left(\begin{array}{rrrr}1&-1&0&0\\1&0&-1&0\\1&0&0&-1 \end{array} \right).
 $$
 This yields an sos decomposition of length $3$ for $ 1-f_p(\xi)^* f_p(\xi)$ on $\T^2$,
namely
$$
1-f_p(\xi)^* f_p(\xi) =q_0(\xi)^* q_0(\xi) ,\qquad q_0(\xi)=H_0v(\xi)=
 \frac{1}{4} \left(\begin{array}{c}  1-\xi_1\\1-\xi_2\\1-\xi_1 \xi_2
  \end{array} \right).
$$
 It also allows us to extend the vector-function $f_p$ to an inner function
 $$
 f(\xi):= \left(\begin{array}{c} f_p \\ q_0 \end{array} \right)(\xi)=
  \frac{1}{4} \left(\begin{array}{c} 1+\xi_1 \xi_2\\1+\xi_2\\1+\xi_1 \\2
    \\ \hline 1-\xi_1\\1-\xi_2\\1-\xi_1\xi_2
  \end{array} \right)
 $$
and have the bilinear representation
$$
1-f_p(\eta)^* f_p(\xi) =q_0(\eta)^* q_0(\xi) +(1-\bar{\eta}_1\xi_1)
q_1(\eta)^* q_1(\xi)+(1-\bar{\eta}_2\xi_2)
q_2(\eta)^* q_2(\xi),
$$
where
$$
   q_1(\xi)=\sqrt{A_1} v_1(\xi)=\frac{1}{2},\qquad
   q_2(\xi)=\sqrt{A_2} v_2(\xi)=\frac{1}{\sqrt{8}}
    \left(\begin{array}{c} 1\\ \xi_1 \end{array} \right).
$$

Let $g= \left(\begin{array}{c} q_1\\ q_2 \end{array} \right)$.
Next, we use the result of Corollary \ref{cor:structure_D}, and derive the following $ABCD$-decomposition of the inner function
 $$
  f(\xi)=
    A+BE(\xi)\left(I-DE(\xi)\right)^{-1}C,
    \quad E(\xi)=\hbox{diag}(\xi_1,\xi_2,\xi_2),
 $$
 where the block matrix $$ \left( \begin{array}{cc} A & B\\
                            C & D\\
           \end{array} \right)
                    =\frac{1}{4} \left(\begin{array}{c|ccc}
                    1&0&0&2\sqrt{2} \\1&0&2\sqrt{2}&0\\1&2&0&0 \\2&0&0&0 \\
                    1&-2&0&0\\1&0&-2\sqrt{2}&0\\1&0&0&-2\sqrt{2}\\
  \hline
  2 &0&0&0\\ \sqrt{2}&0&0&0 \\ 0&2\sqrt{2}&0&0 \end{array}\right)
$$
is an isometry (see subsection \ref{subsec:multi_linear_systems}
from Appendix for details). The blocks were computed by solving the system
 $$
  \left(\begin{array}{cc}A&B\\C&D\end{array}\right)
    \left(\begin{array}{c}1\\ E(\xi)g(\xi)\end{array}\right)
    =\left(\begin{array}{c}f(\xi)\\g_(\xi)\end{array}\right).
 $$
 Thus, we immediately have $A=f(0)$ and $C=g(0)$.
 Moreover, $g(\xi)=C+DE(\xi)g(\xi)$ uniquely determines  $D$, and, likewise,
$f(\xi)=A+BE(\xi)g(\xi)$ uniquely determines  $B$.
\end{Example}

If we apply the construction of \cite{LS06} for the definition of tight wavelet frames,
we will obtain $7$ trigonometric polynomials $a_1,\ldots,a_7$ which satisfy the UEP
for the given trigonometric polynomial $p$ in Example \ref{ex:B111}.
We next show that, by using a shorter extension to an inner function
by only $2$ additional polynomials, and in
combination with the adjunction formula of
Proposition \ref{prop:flip_flop}, we reduce the number of
trigonometric polynomials to $N=5$. Moreover, all the corresponding frame generators have
small support in $[0,2]^2$ and every mask has at most $7$ nonzero coefficients.
Hereby, we improve the existing constructions of tight wavelet frames
for the three-directional box-spline $B_{111}$
in \cite{CS,CH01,LS06}, where $6$ generators with larger support were constructed.

\begin{Example} \label{ex:B111a}
We let $M=2I$, $m=4$, and
 $$
 f_p(\xi)=\left( m^{1/2} \tilde{p}_\chi(\xi) \right)_{\chi \in G'}=
  \frac{1}{4} \left(\begin{array}{c} 1+\xi_1\xi_2 \\1+\xi_2\\1+\xi_1 \\2 \end{array} \right)
 $$
as in Example \ref{ex:B111}.

First, we make use of
 \cite[Example 5.2]{LS06} and
choose another extension of $f_p$ to an inner function  by
only $2$ polynomials (rather than $3$ in Example \ref{ex:B111}), namely
 $$
 \tilde f(\xi):= \left(\begin{array}{c} f_p \\ \tilde q_0 \end{array} \right)(\xi)=
  \frac{1}{4} \left(\begin{array}{c} 1+\xi_1 \xi_2\\1+\xi_2\\1+\xi_1 \\2
    \\ \hline \frac{\sqrt{6}}{2}(1-\xi_1)\\ \frac{\sqrt{2}}{2}(2-\xi_2-\xi_1\xi_2)
  \end{array} \right).
 $$
Simple computation yields
 $$
  1-f_p(\eta)^* f_p(\xi)-\tilde q_0(\eta)^* \tilde q_0(\xi)
    =
 \frac{1}{4}(1-\xi_1\bar{\eta}_1)+ \frac{1}{32}(1-\xi_2\bar{\eta}_2)
(3+\xi_1+\bar{\eta}_1+3\xi_1\bar{\eta}_1).
 $$
Factorization of the (non-diagonal) semi-definite matrices
$$
  \tilde A_1=\frac{1}{4}\quad \hbox{and} \quad
  \tilde A_2=\frac{1}{32}\begin{pmatrix} 3&1\\1&3\end{pmatrix}
 $$
leads to the bilinear representation
$$
1-f_p(\eta)^* f_p(\xi) =\tilde q_0(\eta)^* \tilde q_0(\xi) +(1-\bar{\eta}_1\xi_1)
\tilde q_1(\eta)^* \tilde q_1(\xi)+(1-\bar{\eta}_2\xi_2)
\tilde q_2(\eta)^* \tilde q_2(\xi),
$$
where
$$
   \tilde q_1(\xi)=\frac{1}{2},\qquad
   \tilde q_2(\xi)=\frac{1}{8}
    \left(\begin{array}{c} 2(1+\xi_1)\\ \sqrt{2}(1-\xi_1) \end{array} \right).
$$
The same steps as in Example \ref{ex:B111} give the following
$ABCD$-decomposition of the inner function
 $$
  \tilde f(\xi)=
    A+BE(\xi)\left(I-DE(\xi)\right)^{-1}C,
    \quad E(\xi)=\hbox{diag}(\xi_1,\xi_2,\xi_2),
 $$
 with isometric block matrix
\begin{equation}\label{eq:B111abcd} \left( \begin{array}{cc} A & B\\
                            C & D\\
           \end{array} \right)
                    =\frac{1}{8} \left(\begin{array}{c|ccc}
                    2&0&4&-4\sqrt{2} \\2&0&4&4\sqrt{2}\\2&4&0&0 \\4&0&0&0 \\
                    \sqrt{6}&-2\sqrt{6}&0&0\\ 2\sqrt{2}&0&-4\sqrt{2}&0\\
  \hline
  4 &0&0&0\\ 2&4&0&0 \\ \sqrt{2}&-2\sqrt{2}&0&0 \end{array}\right).
\end{equation}

Next we use the adjunction formula in subsection \ref{subsec:adjunction_formula}
in order to construct the bilinear decomposition
\begin{equation}\label{eq:B111u}
   I_4-f_p(\xi)f_p( \eta)^*=
    u_0(\xi)u_0( \eta)^*+
    (1-\bar{\eta}_1\xi_1)u_1(\xi)u_1( \eta)^*+
    (1-\bar{\eta}_2\xi_2)u_2(\xi)u_2( \eta)^*.
\end{equation}
For this purpose, we cut the last two rows of $A$ and $B$ in \eqref{eq:B111abcd},
leaving the contractive block matrix
\begin{equation}\label{eq:B111abcdnew} \left( \begin{array}{cc} \tilde A & \tilde B\\
                            C & D\\
           \end{array} \right)
                    =\frac{1}{8} \left(\begin{array}{c|ccc}
                    2&0&4&-4\sqrt{2} \\2&0&4&4\sqrt{2}\\2&4&0&0 \\4&0&0&0 \\
  \hline
  4 &0&0&0\\ 2&4&0&0 \\ \sqrt{2}&-2\sqrt{2}&0&0 \end{array}\right)
\end{equation}
which represents
\begin{equation}\label{eq:fpnew}
   f_p(\xi)=\tilde A+\tilde B E(\xi)(I-DE(\xi))^{-1} C.
\end{equation}
By the adjunction formula, we obtain
\begin{equation}\label{eq:fpstar}
   f_p^*(\xi)=f_p(\bar\xi)^*=\tilde A^*+C^*E(\xi)u^*(\xi),
\end{equation}
where
$$
  u(\xi)= \tilde B(I-E(\xi)D)^{-1}=
    \frac{1}{2}
    \begin{pmatrix}
    \xi_2&1&-\sqrt{2}\\ 0&1&\sqrt{2}\\ 1&0&0 \\0&0&0
    \end{pmatrix}.
$$
The polynomial map $u=(u_1,u_2):\C\to\C^{4\times 3}$ defines the functions
$u_1$ (first column) and $u_2$ (last two columns) in \eqref{eq:B111u}.
It remains to construct $u_0$. The representation \eqref{eq:fpstar}
refers to the contractive $ABCD$-matrix
$$
   \begin{pmatrix}
    \tilde A^*&C^*\\ \tilde B^*&D^*
    \end{pmatrix}= \frac{1}{8}\left(\begin{array}{cccc|ccc}
    2&2&2&4&4&2&\sqrt{2}\\ \hline
    0&0&4&0&0&4&-2\sqrt{2}\\  4&4&0&0&0&0&0\\-4\sqrt{2}&4\sqrt{2}&0&0&0&0&0
    \end{array}\right).
$$
An extension of this matrix
to an isometry is obtained by simple linear algebra,
adding the following $5$ rows
$$
   \begin{pmatrix}
    T_0&T_1
    \end{pmatrix}=
    \frac{1}{24}\left(\begin{array}{cccc|ccc}
    6\sqrt{3}&6\sqrt{3}&-2\sqrt{3}&-4\sqrt{3}&-4\sqrt{3}&-2\sqrt{3}&-\sqrt{6}\\
    0&0&-12\sqrt{2}&0&0&12\sqrt{2}&0\\
    0&0&0&12\sqrt{2}&-12\sqrt{2}&0&0\\
    0&0&4\sqrt{6}&-4\sqrt{6}&-4\sqrt{6}&4\sqrt{6}&4\sqrt{3}\\
    0&0&0&0&0&0&12\sqrt{3}
    \end{array}\right).
$$
This extension provides the polynomial map $u_0^*(\xi)=T_0+T_1E(\xi) u^*(\xi)$,
and hence
\begin{eqnarray*}
  u_0(\xi)&=&T_0^*+u(\xi)E(\xi)T_1^*\\ &=&
    \frac{1}{12}\begin{pmatrix}
    \sqrt{3}(3-\xi_1\xi_2)&  3\sqrt{2}\xi_2& - 3\sqrt{2}\xi_1\xi_2 & -\sqrt{6}\xi_1\xi_2& -3\sqrt{6}\xi_2\\
    \sqrt{3}(3-\xi_2)&  3\sqrt{2}\xi_2& 0 & 2\sqrt{6}\xi_2& 3\sqrt{6}\xi_2\\
    -\sqrt{3}(1+\xi_1)& -6\sqrt{2}& - 3\sqrt{2}\xi_1 & \sqrt{6}(2-\xi_1)& 0\\
-2\sqrt{3}& 0   & 6\sqrt{2}&-2\sqrt{6} &0
\end{pmatrix}.
\end{eqnarray*}
Finally, the restriction of $(f_p,u_0)$ to $\T^2$
defines the matrix $U(\xi)$, $\xi=(z_1^2,z_2^2)$, in the UEP identities
$$
I_4-f_p(\xi)f_p(\xi)^*=U( \xi)U(\xi)^*,\qquad \xi\in\T^2.
$$
Hence, the number of columns of $u_0$  and the
degree of $u_0$ determine the number of framelets and their support.
We obtain the following $5$ trigonometric polynomials
$$
    \begin{pmatrix}
    a_1(z)\\ \vdots\\ a_5(z)
\end{pmatrix},
=\frac{1}{24}
    \begin{pmatrix}
    \sqrt{3}(3+3z_1-z_2-2z_1z_2-z_1^2z_2-z_1z_2^2-z_1^2z_2^2)\\
    -3\sqrt{2}(2z_2-z_2^2-z_1z_2^2)\\
    3\sqrt{2}(2z_1z_2-z_1^2z_2-z_1^2z_2^2)\\
    \sqrt{6}(2z_2-2z_1z_2-z_1^2z_2+2z_1z_2^2-z_1^2z_2^2)\\
    -3\sqrt{6}(z_2^2-z_1z_2^2)
\end{pmatrix}.
$$
\end{Example}

\section{Appendix: Multivariate system analysis} \label{appendix}

The investigation of the mask $p$ of a tight wavelet frame
naturally brings into the picture the class of complex polynomials
with a prescribed bound in the polydisk $\D^d$. Their structure
can be better understood from the more general perspective of
bounded analytic functions in the polydisk. Fortunately, there is
a great deal of accumulated knowledge on this topic, especially
arising from a remarkable connection to multivariate system
analysis. Without aiming at completeness, the present appendix
offers a quick introduction to the subject. The results listed
below are used in section \ref{sec:system_theory_twf}.
\bigskip

\subsection{Single variable}
 We collect below some classical results which provide the
starting point for the more intricate structure of bounded analytic functions in the polydisk.

Let $f(z), | f(z)| \leq 1, $ be an analytic function defined in the disk $\D = \{z \in \C; |z|<1\}$. Leaving the case of a constant function aside, we can assume
that $|f(z)|<1$ in the disk, and define the function $g(z) =  \frac{1 + f(z)}{1-f(z)},$ so that $\Re g(z) \geq 0$ for all $|z|<1$. Let
$g_r(z) = g(rz), \ 0 <r<1,$ so that the functions $g_r$ are defined in a neighborhood of the closed disk and $\lim_{r \rightarrow 1} g_r = g$
uniformly on compact subsets of $\D.$ A direct application of Cauchy's formula yields:
$$ g_r(w) = \int_{-\pi}^\pi \frac{e^{i\theta} + w}{e^{i\theta} - w} \frac{\Re g_r (e^{i\theta}) d\theta}{2 \pi} + i \Im g(0).$$
Remark that the measures $d \mu_r = \frac{\Re g_r (e^{i\theta}) d\theta}{2 \pi}$ are non-negative, of uniform mass equal to $\Re g(0)$, hence they form a
compact set in the weak-$*$ topology of measures on the unit torus. By passing to a limit point we obtain a positive measure $\mu$ with the property
\begin{equation}\label{RH}  g(w) = \int_{-\pi}^\pi \frac{e^{i\theta} + w}{e^{i\theta} - w} d\mu(\theta) + i \Im g(0).\end{equation}
Since the trigonometric polynomials are dense in the space of continuous functions on the torus, we infer that the measure
$\mu$ is unique with the above property.

 Formula
(\ref{RH}) is known as the {\it Riesz-Herglotz representation} of all analytic functions with non-negative real part in the disk. Since $\D$ is simply connected,
for any harmonic function $u:\D \longrightarrow \R$ there exists an analytic function $g :\D \longrightarrow \C$ such that $u = \Re g$.
Putting together these observations we have proved the equivalence between the first two statements in the next theorem.

\begin{Theorem}[Riesz-Herglotz] \label{RHT} Let $g : \D \longrightarrow \C$ be an analytic function. The following assertions are equivalent:

a). $\Re g \geq 0$;

b). There exists a positive measure $\mu$ on $\T = \partial \D$, such that (\ref{RH}) holds;

c). The kernel $\frac{g(z) + \overline{g(w)}}{1-z\overline{w}}$ is positive semi-definite on $\D \times \D$.
\end{Theorem}

\begin{proof} $a) \Rightarrow b)$ was proved before. If $b)$ holds true, then
$$ \frac{g(z) + \overline{g(w)}}{1-z\overline{w}} = 2 \int_{-\pi}^\pi \frac{d\mu(\theta)}{(e^{i\theta}-z)(e^{-i\theta}-\overline{w})},$$
whence $c)$ is true. Finally,  $c) \Rightarrow a)$ because a positive semi-definite kernel has
non-negative values on the diagonal.
\end{proof}

It is important to note that {\it any} positive measure $\mu$ on the one-dimensional torus $\T$ can arise in the Riesz-Herglotz parametrization.
Also remark that, from $g(z) =  \frac{1 + f(z)}{1-f(z)}$ we infer
$$ 2 \frac{1-f(z) \overline{f(w)}}{1-z\overline{w}} = \frac{1}{1-f(z)}\frac{g(z) + \overline{g(w)}}{1-z\overline{w}}\frac{1}{1-\overline{f(w)}},$$
hence:

\begin{Corollary} An analytic function $f$ maps the unit disk into itself if and only if the
kernel  $\frac{1-f(z) \overline{f(w)}}{1-z\overline{w}}$ is positive semi-definite, that means
$$ \frac{1-f(z) \overline{f(w)}}{1-z\overline{w}} = \sum_{j=1}^N h_j(z) \overline{h_j(w)},$$
where $h_j(z)$ are analytic functions in the disk, and $N \leq \infty$.
\end{Corollary}

Note that, even if $f(z)$ a polynomial, the factors $h_j$ may not be polynomials, or $N$ may be equal to infinity. On one hand, one can
factor $f(z) = f_1(z) f_2(z)$ and use induction based on the identity:
$$ \frac{1- f_1(z) f_2(z) \overline{f_1(w) f_2(w)}}{1-z\overline{w}} =  \frac{1- f_1(z)  \overline{f_1(w) }}{1-z\overline{w}} + f_1(z)  \frac{1-  f_2(z) \overline{f_2(w)}}{1-z\overline{w}} \overline{f_1(w)},$$
having to deal in the end only with a linear factor. But then, even for a constant function $f(z) = c$, the decomposition
$$ \frac{1-|c|^2}{1-z\overline{w}} = (1-|c|^2) \sum_{j=0}^\infty z^j \overline{w}^j$$ contains infinitely many terms.

The natural framework for finitely determined decompositions of the above type
is realized by a class of rational functions, appearing in the celebrated Schur
algorithm, see for instance \cite{FF}.
\bigskip

\subsection{Several variables} \label{subsec:Appendix_several_variables}
The analogue of Riesz-Herglotz formula exists in several
variables, in general on polyhedral or homogeneous domains. The
case of the polydisk was studied by Koranyi and Pukansky
\cite{KP}. For instance they proved that an analytic function
$f(z), z \in \D^d,$ is uniformly bounded $( |f(z)| \leq C, \  z
\in \D^n)$ if and only if the hermitian kernel
$$ \frac{C^2 - f(z)\overline{f(w)}}{\prod_{k=1}^d (1-z_k \overline{w_k})}$$
is positive semidefinite. When compared to the single variable case, this formula turns out to be of limited importance for the expected applications. For instance the celebrated Nevanlinna-Pick
interpolation theorem does not hold for this positive definite kernel, see \cite{AM}.

The subtle distinction between $1$D and $d$D with $d \geq 2$ comes from a celebrated result of von Neumann. To be more precise, let $T \in L(H)$
be a linear bounded contraction $\|T\| \leq 1$ acting on a complex Hilbert space. Let $f(z)$ be a strictly contractive analytic function in the disk.
A direct consequence of Riesz-Herglotz formula yields
$$ I - f(rT)^\ast f(rT) = $$ $$
(I - f(rT)^\ast)^{-1} \int_{\T} (I-\overline{u}T^\ast)^{-1} (I-T^\ast T) (I-{u}T)^{-1} d\mu(u) (I - f(rT))^{-1} \geq 0,$$
where the positivity is in the sense of Hilbert space operators. By passing to limit with $r \rightarrow 1$ and allowing $f$ to be contractive we obtain
{\it von Neumann's inequality}:\\

{\it For every analytic function $f$ defined in a neighborhood of the closed unit disk and Hilbert space contractive operator $T$ one has
$$ \| f(T) \| \leq \|f\|_{\infty, \D}.$$}

Due to an observation of Ando, the above inequality remains true for the bi-disk $\D^2$; but fails for $\D^d$ with $d \geq 3$, see \cite{AM}. The contractive analytic functions $f$
in $\D^d$ which satisfy the multi-variate analogue of von-Neumann inequality
$$ \| f(T)\| \leq \|f \|_{\infty, \D^d},$$
for every commutative tuple $T=(T_1,...,T_d)$ of Hilbert space contractions form the {\it Schur-Agler class} of functions. Examples of contractive functions
in $\D^d, d \geq 3,$ which do not belong to this class were known for a long time, see \cite{AM}. For instance, the following homogeneous polynomial
in three variables
\begin{equation} \label{Druryex}
g(z_1, z_2, z_3) = z_1^3 + z_2^3 + z_3^3 - 3z_1 z_2 z_3
\end{equation}
satisfies
$$ \| g \|_{\infty, \D^3} = 3 \sqrt{3},$$
but there exists a commuting triple $(T_1,T_2,T_3)$ of linear contractions acting on a 8 dimensional Hilbert space, so that
$$ \|g(T_1,T_2,T_3)\| = 6.$$ For details see \cite{Drury}.

A constructive approach revealing the structure of
Schur-Agler functions was completed only during the last decade. The next section
collects some results in this direction.\bigskip

\subsection{Multivariate linear systems} \label{subsec:multi_linear_systems}

The theory of bounded analytic functions in the disk had much to
gain from a natural connection with the control theory of linear
systems. The resulting interdisciplinary field was vigorously
developed during the last forty years, with great benefits for
both sides. The multivariate aspects of bounded analytic functions
(say in the polydisk) seen as transfer functions of linear systems
with multi-time dependence were revealed only during the last
decade, see \cite{BtH} for an excellent survey. We reproduce below
a few fundamental facts of interest for the present work. We deal
exclusively with a state-space formulation, with the explicit
purpose of parametrizing the polynomials (or analytic functions)
we are interested in by structured block-matrices.

The starting point is a quadruple of linear bounded Hilbert space operators $\{ A,B,C,D\}$, acting, as a block matrix on two direct sums of Hilbert spaces:
$$ \left( \begin{array}{cc} D & C\\
                                             B & A\\
                                             \end{array} \right) : \begin{array}{c} H\\
                                             \oplus\\
                                                                                                            X\end{array} \longrightarrow \begin{array}{c}
                                                                                                               H \\
                                                                                                               \oplus \\
                                                                                                               Y \end{array}.$$
Moreover, we decompose $H = H_1\oplus...\oplus H_d$ into a direct
sum and consider the finite difference scheme
      $$ \left( \begin{array}{c} h_1(\alpha + e_1)\\
                                                  \vdots \\
                                                  h_d(\alpha + e_d)\\ \end{array} \right) = D     \left( \begin{array}{c} h_1(\alpha)\\
                                                  \vdots \\
                                                  h_d(\alpha)\\ \end{array} \right)     + C u(\alpha),$$
       $$ y(\alpha) = B h(\alpha) + A u(\alpha), \ \ \ \ \alpha \in \N^d.$$     Above $e_k$ are the generators of the semigroup $\N^d$. In linear system theory language, $u(\alpha)$
       is the input vector, $h(\alpha)$ is the state space vector and $y(\alpha)$ is the output. All vectors running in the respective Hilbert spaces, with $\N^d$ as a multi-time semigroup.
       Let $E(z) = z_1 I_{H_1} \oplus ...\oplus z_d I_{H_d} : H \longrightarrow H$ be regarded as a diagonal operator whose diagonal entries  dependent lineraly on  $z_1, \ldots,z_d \in \C$.
       A great deal of stability analysis of the above finite difference scheme can be read from the associated {\it transfer function}:
       $$ F(z) = A + BE(z) (I - DE(z))^{-1} C,$$
       first defined for small values of $|z|$, and then analytically continued as far as possible. In case the state space $H$ is finite dimensional $D$ is a finite matrix and hence the transfer function is (vector valued) rational.

       The remarkable result which establishes the bridge between contractive analytic functions in the polydisk and linear system theory can be stated as follows, as a combination of an
       older theorem of Agler (see \cite{AM}) and a more recent one due to Ball and Trent \cite{BT}.

\begin{Theorem}[Agler, Ball, Trent] \label{th:Agler_Ball_Trent} Let $X,Y$ be Hilbert spaces
and let $f: \D^d \longrightarrow L(X,Y)$ be an analytic function.
The following are equivalent.

a). For every commutative tuple $T=(T_1,...,T_d)$ of linear
contractive operators acting on a Hilbert space $K$, von-Neumann's
inequality
$$ \sup_{\epsilon_k<1}   \| f(\epsilon_1 T_1,..., \epsilon_d T_d) \| \leq 1,$$
holds;

b). There exist auxiliary Hilbert spaces $H_k$ and analytic
functions $L_k : \D^d \longrightarrow L(X, H_k)$ such that
$$ I - f(w)^\ast f(z) = \sum_{k=1}^d (1-\overline{w_k}z_k) L_k(w)^\ast L_k(z), \ \ \ z, w \in \D^d;$$

c). There exists an auxiliary Hilbert space $H = H_1 \oplus ...\oplus H_d$ and a unitary operator
$$ \left( \begin{array}{cc} A & B\\
                            C & D\\
           \end{array} \right) : \begin{array}{c} X\\
           \oplus\\
           H\end{array} \longrightarrow \begin{array}{c} Y \\
           \oplus \\
           H \end{array},
$$
such that
$$ f(z) = A + BE(z) (I-DE(z))^{-1} C, \ \ z \in \D^d.$$
\end{Theorem}

The meaning of $f(\epsilon_1 T_1,..., \epsilon_d T_d) \in L(X,Y) \otimes L(K)$ can be made precise by the Riesz-Dunford functional calculus, or by
a formal substitution of $z_k$ by $T_k$ in a power series expansion of the function $f$. Remember that
$E(z) = z_1 I_{H_1} \oplus ...\oplus z_d I_{H_d} : H \longrightarrow H$ is a diagonal operator, linear in the variables $z$.

In practice it is sometimes useful to relax condition c) by asking only that the $2 \times 2$ block operator is contractive. In this case, denoting
$g(z) = (I-DE(z))^{-1}C$, or equivalently $g(z) = C + DE(z)g(z)$,
we find
$$ \left( \begin{array}{cc} A & B\\
                                             C & D\\
                                             \end{array} \right) \left( \begin{array}{c} I\\
                                             E(z) g(z) \end{array} \right) = \left( \begin{array}{c}
                                             f(z) \\
                                              g(z) \end{array} \right).
 $$
                                              Thus,
                                                                                                               $$ \| f(z)\|^2 + \|g(z)\|^2 \leq 1 + \| E(z) g(z)\|^2 \leq 1 + \|g(z)\|^2
                                                                                                               $$
and, therefore, $\|f(z)\| \leq 1$ for all $z \in \D^d$. Since every contractive operator admits a unitary dilation (with the price of increasing the Hilbert space $H$), all functions $f$
constructed above (from a contractive block operator) belong to Schur-Agler's class.

A constructive approach for determining the
$A,B,C,D$ matrices from a function $f(z)$ appears in an early
article by Kummert \cite{Kummert}. See also the monograph \cite{Bose} and the D-module approach to such questions of system theory proposed in \cite{Zerz}.

For the functions $f$ that extend to the closed polydisk, we derive the following defect (from unity) formula.
Due to $E(z)^*E(z)=I$ for $z \in \T^d$, we have, for $z \in \T^d$,
$$ I-f(z)^\ast f(z) = \left( \begin{array}{c} I\\
                                             E(z) g(z) \end{array} \right)^\ast \left( \begin{array}{c} I\\
                                             E(z) g(z) \end{array} \right) -  \left( \begin{array}{c}
                                             f(z) \\
                                             g(z) \end{array} \right)^\ast \left( \begin{array}{c}
                                             f(z) \\
                                             g(z) \end{array} \right)
                                              = $$ $$  \left(T \left( \begin{array}{c}
                                              I \\
                                              E(z) g(z) \end{array} \right)\right)^\ast \left(T \left( \begin{array}{c}
                                              I \\
                                              E(z) g(z) \end{array} \right)\right),
$$
where
\begin{equation}\label{def:T}
I - T^\ast T =     \left( \begin{array}{cc} A & B\\
                                             C & D\\
                 \end{array} \right)^\ast  \left( \begin{array}{cc} A & B\\
                  C & D\\
                  \end{array} \right).
\end{equation}

Summing up, we are led to the following result.

\begin{Theorem} \label{completion to SOS}  Let $f(z) \in L(X,Y)$ be an operator valued polynomial map belonging to the Schur-Agler class. If a contractive block-matrix realization
of $f(z)$ exists and $(I-DE(z))^{-1}C$ is a polynomial, then there are Hilbert spaces $Y_0,Y_1, \ldots, Y_d$ and polynomial maps $q_j(z) \in L(X,Y_j)$ with the property
$$
I - f(w)^\ast f(z) = q_0(w)^\ast q_0(z) + \sum_{k=1}^d (1- \overline{w_k} z_k) q_k(w)^\ast q_k(z), \ \ \ z \in \C^d.$$
   \end{Theorem}

   \begin{proof} Let us split the matrix $T$ in \eqref{def:T} as $T = (A_0, B_0) : X \oplus H \longrightarrow Y_0$ where $Y_0$ is an auxiliary Hilbert space of dimension not exceeding the rank of $T$. Consequently, the matrix
   $$ V = \left( \begin{array}{cc}
                     A_0 & B_0\\
                     A&B\\
                     C&D  \end{array} \right)$$
is isometric. Hence
$$ \left( \begin{array}{c}
  q_0\\
 f \end{array} \right) =  \left( \begin{array}{c}
  A_0\\
  A \end{array} \right)        + \left( \begin{array}{c}
  B_0\\
  B \end{array} \right)     E(z)(I - DE(z))^{-1} C  $$
  is an operator valued polynomial  map. Since $V$ is an isometry, denoting $g(z) = (I-DE(z))^{-1}C$, we find
$$
 \| q_0(z) \|^2 + \| f(z)\|^2 + \|   g(z)\|^2 =1+ \| E(z)g(z)\|^2
$$
and the conclusion follows by polarization.
  \end{proof}

In the case $f(z)$ is an inner polynomial, i.e. $\|f(z)\|=1$ on $\T^d$, the second condition in the statement of Theorem \ref{completion to SOS} is automatically satisfied. This fact is stated in the following result proved in \cite{CW}.

\begin{Theorem}[Cole-Wermer] \label{th:CW}
Assume that $X,Y$ are finite dimensional Hilbert spaces and $f(z) \in L(X,Y)$ is a polynomial in the Schur-Agler class such that $f(z)^\ast f(z) = I$ for all $z \in \T^d$. Then
$$
 I - f(w)^\ast f(z) = \sum_{k=1}^d (1- \overline{w_k} z_k) q_k(w)^\ast q_k(z),
$$
where $Y_k$ are finite dimensional Hilbert spaces, $q_k(z) \in L(X, Y_k), 1 \leq k \leq d,$ are
polynomial maps, and
$$
 \max( \deg q_1,\ldots, \deg q_d) < \deg f.
$$
Moreover, the spaces $Y_k$ can be chosen so that  $\dim Y_k \le \dim X \times \dim \C_{\deg f-1}[z]$.
\end{Theorem}

\begin{proof} Fix  $\omega \in \T^d$ and restrict the Agler's decomposition to the ray pointing at
$\omega$
\begin{equation}\label{radial}
I - f(r \omega)^\ast f(r \omega) = (1-r^2) \sum_{k=1}^d q_k(r\omega)^\ast q_k(r \omega), \ \ r <1.
\end{equation}
Since $f(\omega)^\ast f(\omega) = I$, the quotient
\begin{equation}\label{eq:CW}
\frac{I - f(r \omega)^\ast f(r \omega) }{1-r^2} = \frac{1}{1+r}[ f(\omega)^\ast \frac{f(\omega)-f(r\omega)}{1-r}+
                              \frac{f(\omega)^\ast- f(r\omega)^\ast}{1-r} f(r \omega) ]
\end{equation}
is a rational function without poles on the positive semi-axis, with polynomial growth
at infinity of order $2\deg p -2$. The power expansion at zero of the factors $q_k$ is
$$
q_k(r\omega) = \sum_{\alpha \in \N^d} q_{k,\alpha} r^{|\alpha|} \omega^\alpha, \quad
q_{k,\alpha}: X \rightarrow Y_k,
$$
with convergence assured for $0 \leq r <1$.

Next we free $\omega \in \T^d$ and consider the zero-th order Fourier coefficient of the decomposition
(\ref{radial}). The right hand side of \eqref{eq:CW}  is an analytic function in $r$ with
$$
\sum_{k+1}^d \sum_{\alpha \in \N^d} q_{k,\alpha}^\ast q_{k,\alpha} r^{2|\alpha|}
$$
convergent for $r<1$. Moreover, it is rational on the semi-axis $r \in [0,\infty)$
with denominator $(1+r)$ and with polynomial growth
at infinity of order $2\deg p -2$.  Hence, the analytic function, as the function of $r$,
$$
(1+r) \sum_{k+1}^d \sum_{\alpha \in \N^d} q_{k,\alpha}^\ast q_{k,\alpha} r^{2|\alpha|}
$$
is a polynomial. As all its coefficients $q_{k,\alpha}^\ast q_{k,\alpha}$ are non-negative,
we conclude that the operator coefficients $q_{k, \alpha}$ vanish for $|\alpha| \geq \deg p$. This
allows us to choose
which implies that $\dim Y_k \leq \dim X  \cdot \dim \C_{\deg p-1} [z]$.
\end{proof}

\noindent For more details on the above proof and its immediate implications we refer to \cite{CW}.
The following consequence of Theorem \ref{th:CW} is of interest. It states that in the univariate case
the matrix $D$ is nilpotent.

\begin{Corollary} \label{cor:general_D_nilpotent} Assume that $f \in L(\C,\C^m)$ is a polynomial and $f(z)^*f(z)=1$ on $\T^1$. Then
there exists a realization $ f(z)=A+zB(I-zD)^{-1} C$
with an isometry
$$ \left( \begin{array}{cc} A & B\\ C & D\\ \end{array} \right) : \begin{array}{c} \C\\ \oplus\\
 \C^n \end{array} \longrightarrow \begin{array}{c} \C^{m} \\ \oplus \\ \C^n \end{array}, \quad
 n \le \deg f,
$$
and $\hbox{det}(I-zD)^{-1}=1$.
\end{Corollary}
\begin{proof}
 As we deal with the univariate case, we know that $f$ belongs to Schur-Agler class. This implies the existence
 of a corresponding minimal $ABCD$-representation. By Theorem \ref{th:CW},
$$
 q_1(z)=(I-zD)^{-1}C=\sum_{\alpha=0}^{\deg f-1} q_1(\alpha) z^\alpha, \quad q_1(\alpha): \C \rightarrow Y_1,
$$
where
$$
 Y_1=\hbox{span} \{q_1(\alpha) \ : \ \alpha=0, \ldots, \deg f-1\}=\C^n, \quad n\le \deg f.
$$
Then the identity $zD q_1(z)=q_1(z)-q_1(0)$ implies
\begin{eqnarray*}
 D q_1(\alpha)&=&q_1(\alpha+1), \quad \alpha=0, \ldots, n-2, \\
 D q_1(n-1)&=&0.
\end{eqnarray*}
Therefore, $D$ is nilpotent and $D^{n}=0$.
\end{proof}

 \subsection{Adjunction formula}  \label{subsec:adjunction_formula}
The class of Schur-Agler functions in the polydisk is closed under Hilbert space conjugation.
The simple adjunction formula below has direct implications to tight wavelet frames, as seen in the body of the present article.

 \begin{Proposition}\label{prop:flip_flop} Let $p(z) \in L(X,Y)$ be a polynomial in the Schur-Agler class of the polydisk $\D^d$:
 $$ p(z) = A + BE(z) (I-DE(z))^{-1})C.$$ Write $p^\ast(\zeta) = p(\overline{\zeta})^\ast.$ Then
 $$ p^\ast(\zeta) = A^\ast + C^\ast E(\zeta) (I-D^\ast E(\zeta))^{-1} B^\ast.$$

 If, in addition $(I-D^\ast E(\zeta))^{-1} B^\ast$ is a polynomial function, then
 $$ I - p(z) p(w)^\ast = q_0(z) q_0(w)^\ast + \sum_{k=1}^d (1-z_k \overline{w_k}) q_k(z) q_k(w)^\ast$$
 with matrix valued polynomial functions $q_0,\ldots , q_d$.
 \end{Proposition}

\begin{Remark} The assumption that $(I-D^\ast E(z))^{-1} B^\ast$ is a polynomial function
is satisfied e.g. if $\det(I-D^\ast E(z))=1$. Writing $D^\ast E(z)$ as a linear pencil
$$ D^\ast E(z) = D^\ast p_1 z_1 + \cdots + D^\ast p_d z^d, $$
with mutually orthogonal projections that add up to the identity $p_1+p_2 +\cdots+p_d = I$,
we infer that, for every point $z \in \C^d$, the operator $D^\ast E(z)$ is nilpotent.
Indeed, it suffices to consider the equation
$$ \det ( I -\zeta D^\ast E(z)) = 1, \ \  \zeta \in \C,\quad z\in\C^d,$$
and put $D^\ast E(z)$ in the upper-triangular form. By taking adjoints, this amounts to
the condition that the linear pencil $ E(z)D$ is nilpotent for all $z \in \C^d$. By cyclic invariance
$$ \det(I-E(z)D) = \det(I-DE(z))$$
hence the original assumption is equivalent to the fact that the linear pencil $DE(z)$
consists of nilpotent linear transformations.

\end{Remark}

 Greg Knese and collaborators
have recently revealed many new details about Agler's
decomposition of a positive kernel on the polydisk, see
\cite{Knese}.


\end{document}